\documentclass[12pt]{article}
\UseRawInputEncoding
\usepackage{amsmath}
\usepackage{graphicx}
\usepackage{amssymb}
\usepackage{ntheorem}

\usepackage{geometry}
\usepackage{setspace}
\usepackage{fancyhdr}
\numberwithin{equation}{section}
\numberwithin{figure}{section}

\newenvironment{proof}{{\noindent\it \bf Proof}\quad}{\hfill $\square$\par}
\newtheorem{theorem}{Theorem}[section]
\newtheorem{lemma}{Lemma}[section]

\newtheorem{remark}{Remark}[section]
\newtheorem*{acknowledgement}{Acknowledgement}

\allowdisplaybreaks[4]

\begin{document}
	
	\rmfamily 
	
	\title{Symmetric Stationary Boundary Layer}
	\author{ Chen Gao, Liqun Zhang and Chuankai Zhao}
	\date{}
	
	\maketitle
	
	\begin{abstract}
		Considering the boundary layer problem in the case of two-dimensional flow past a wedge with the wedge angle $\varphi=\pi\frac{2m}{m+1}$, Oleinik and Samokhin\cite{OS} obtained the local well-posedness results for $m \geq 1$. In this paper, we establish the existence and uniqueness of classical solutions to the Prandtl systems for arbitrary $m>0$, which solves the steady case in Open problem 6 proposed by Oleinik and Samokhin \cite{OS}. Our proof is based on the maximum principle technique at the Crocco coordinates and the most important observation that when the fluid approaches a sharp point, it seems the self-similar solutions. Then we obtain the existence and uniqueness of the solution with the help of the self-similar solutions by the Line Method. Furthermore, we similarly establish the well-posedness results of three-dimensional flow past a cone.
		~\\
		
	\end{abstract}
	
	\section{Introduction}
	
	It is well-known that the Prandtl system is obtained as a simplification of the Navier-Stokes system and describes the motion of a fluid with small viscosity about a solid body in boundary layer. In this paper, we consider the following Prandtl system for the plane stationary symmetric incompressible flow:
	\begin{equation}\label{ps}
		\begin{gathered}
			u\frac{\partial u}{\partial x}+v\frac{\partial u}{\partial y}=U\frac{dU}{dx}+\nu\frac{\partial^2 u}{\partial y^2},\\
			\frac{\partial u}{\partial x}+\frac{\partial v}{\partial y}=0
		\end{gathered}
	\end{equation}
    in the domain $D=\{0<x<X,0<y<\infty\}$ with the boundary conditions 
    \begin{equation}\label{bd}
    	\begin{gathered}
    		u(0,y)=0,u(x,0)=0,v(x,0)=v_0(x),\\
    		u(x,y)\rightarrow U(x) \quad \text{as} \quad y\rightarrow \infty.
    	\end{gathered}
    \end{equation}

    Where $\nu>0$ is the coefficient of kinematic viscosity, the fluid density $\rho$ is equal to 1, and the function $U(x)$ is a given longitudinal velocity component of the outer flow: $U(0)=0$, $U(x)>0$ for $x>0$.
	
	For instance, when parallel flow flows through a wedge with the wedge angle $0< \varphi<2\pi$, the system described above appears near the edge of the wedge. The flow past a wedge can be regarded as a coupling of two angular flows, and its external flow corresponds to the Euler flow of the angular flow. Then the flow velocity varies according to the law $U(x)\sim Cx^m$ for $\varphi =\pi \frac{2m}{m+1}$ where $C,m=\text{const}>0$ (cf.Figure \ref{fig:wedges}). Perhaps, due to the curvature of the surface, lower-order terms appear in the asymptotics. 

	\begin{figure}[h]
		\centering
		\includegraphics[scale=0.5]{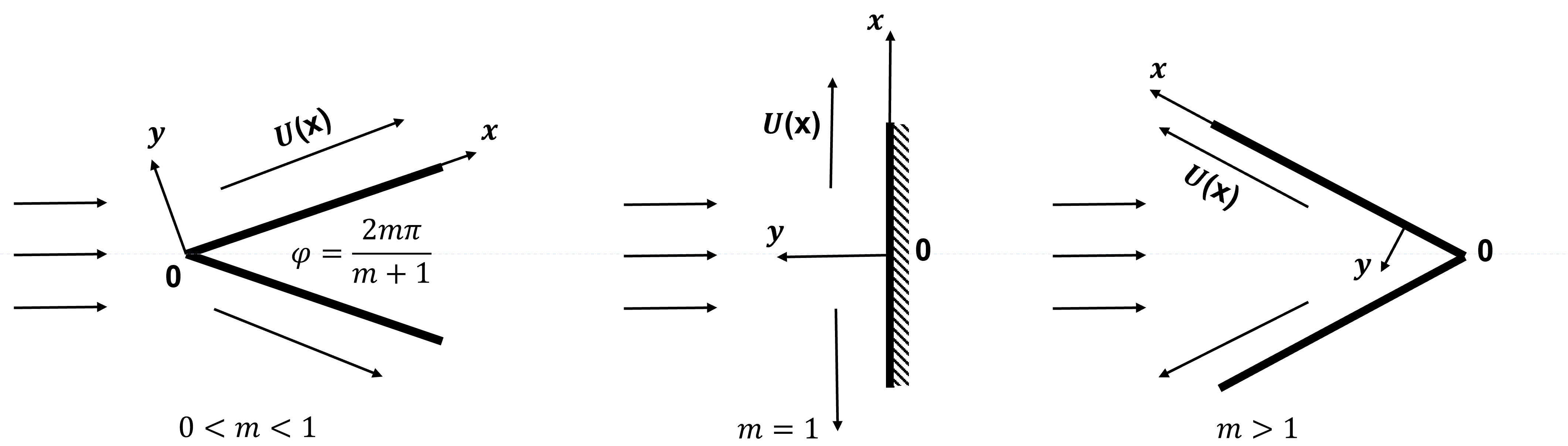}
		\caption{Flow parallel to the symmetry axis of a wedge }
		\label{fig:wedges}
	\end{figure}
	
	There is a lot of literature on theoretical, numerical, and experimental studies on Prandtl’s system, see \cite{OS}\cite{S} and the references therein. In particular, Oleinik and Samokhin \cite{OS} gave a systematic exposition of the main rigorous mathematical results. The main results of Oleinik and Samokhin can be summarized as that, when $m\geq 1$, there exists a unique classical solution to the problem (\ref{ps}), (\ref{bd}) for $x<X$, where $X$ is small, in the class of functions having a certain behavior at infinity with respect to $y$, namely, $U(x)-u(x,y)\sim e^{-\alpha y^2}$ as $y\rightarrow \infty$, where $\alpha=\text{const}>0$. 
	
	It is necessary to mention the well-posedness results of the plate boundary layer equations. For the monotonic data, Oleinik and Samokhin \cite{OS} obtained the local existence of classical solutions of 2D Prandtl equations basing on the Crocco transformation. And Wang and Zhang \cite{WZ1} proved the global-in-$x$ $C^\infty$ regularity up to the boundary $y = 0$ using the maximum principle technique. For non-stationary case, Xin and Zhang \cite{XZ} proved the global existence of weak solutions for the favorable pressure. And then Xin, Zhang and Zhao \cite{XZZ} proved the global existence of smooth solutions in the case of favorable pressure gradient.
	
	There are some results for the asymptotic behavior of the steady Prandtl equation as $x\rightarrow \infty$. Serrin \cite{Se} proved that Oleinik’s solution $u$ converges to the famous Blasius solution $\bar{u}$ in $L_y^\infty$ sense as $x\rightarrow \infty$. Iyer \cite{I} proved the explicit decay estimates of $u-\bar{u}$ and its derivatives when the initial data is a small localized perturbation of the Blasius profile. Wang and Zhang \cite{WZ2} proved that the explicit decay estimate of $u-\bar{u}$ in $L_y^\infty$ sense for general initial data with exponential decay, and also proved that the decay estimates of its derivatives when the data has an additional concave assumption.
	
	Boundary layer separation is also an important concern, and some results have been given by predecessors. The initial work belongs van Dommelen and Shen \cite{DS}. Recently in the case of adverse pressure gradient, Dalibard and Masmoudi \cite{DM} as well as Shen, Wang, and Zhang \cite{SWZ} justified the physical phenomenon of boundary layer separation and studied the local behavior of the solution near the separation point. The relevant results about the unsteady Prandtl equation can be found in \cite{KVW}\cite{WE}.
	
	The main purpose of this paper is to establish the existence and uniqueness of solutions to the problem (\ref{ps}), (\ref{bd}) for arbitrary $m>0$. This generalizes the local well-posedness results for $m\geq 1$ due to Oleinik\cite{OS}. Furthermore, we similarly obtain the well-posedness results to three-dimensional flow past a cone with the cone angle. In the last section, we will describe the work on the three-dimensional symmetry case in more detail.
	
	Our main result in this paper can be stated as 
	\begin{theorem}\label{prandtl}
		Assume that
		\begin{equation}\label{assum}
			\begin{gathered}
				U(x)=x^mV(x),\quad v_0(x)=x^{\frac{m-1}{2}}v_1(x),\\
				V(x)=a+a_1(x),\quad V(0)=a>0,\quad v_1(0)=0,\\
				|a_1(x)|\leq N_1x,\quad |v_1(x)|\leq N_2x.
			\end{gathered}
		\end{equation}
	    where $N_1,N_2=\text{const}>0$. Then the initial boundary value problem (\ref{ps}),(\ref{bd}) in the domain $D=\{0<x<X,0<y<\infty\}$, for some $X$ depending on $U$ and $v_0$, has a unique classical solution $u(x,y), v(x,y)$ with the following properties: $u_y>0$ for $y \geq 0, x>0$ ; $u/U, u_y /(x^{(m-1) / 2} U(x))$ are bounded and continuous in $\bar{D}$ ; $u>0$ for $y>0$ and $x>0 $; $ u(x, y) \rightarrow U(x)$, $u_y \rightarrow 0$ as  $y \rightarrow \infty$; moreover, the solution $u(x,y)$ has the following property as $x\rightarrow 0$:
	    \begin{equation}
	    	u(x,y)\sim x^m f^{'}\left(yx^{\frac{m-1}{2}}\right),
	    \end{equation}
        where $f$ will be mentioned in the following section; and the solution $u(x,y)$ has the following inequalities as $y\rightarrow \infty$:
		\begin{equation}
		    M_{1} \exp (-M_{2} x^{m-1} y^2) \leq 1-\frac{u}{U} \leq M_{3} \exp (-M_{4} x^{m-1} y^2),
	    \end{equation}
        where $M_i=\text{const}>0,i=1,2,3,4$.
	\end{theorem}

    \begin{remark}
    	It is easy to verify that the results in this paper can be generalized to the following situations:
    	\begin{equation}
    		\begin{gathered}
    			v_0(x)=x^{\frac{m-1}{2}}v_1(x),\quad v_1(x)=b+b_1(x),\\
    			b\leq 0,\quad |b_1(x)|\leq N_3x.
    		\end{gathered}
    	\end{equation}
        
        It should be noted that when $m\geq 1$, $v_0(x)\rightarrow 0$ as $x\rightarrow 0$ and $v_0=0$ at the stagnation point. When $0<m<1$ and $b<0$, $v_0(x) \rightarrow -\infty$ as $x\rightarrow 0$, which is physically impossible.
    \end{remark}

    \begin{remark}
    	From the theorem \ref{prandtl}, we note that a point $x_0$, $0<x_0<X$ can be found as the initial location of the continuation problem for boundary layer, where $u(x_0,y)>0$ and $U(x_0)>0$. With the help of the results of the continuation problem for boundary layer mentioned earlier, we can naturally obtain the existence of the Prandtl system (\ref{ps}), (\ref{bd}) for any $X>0$ in the case of favorable condition either $U_x\geq 0$ and $v_0(x)\leq 0$ or $U_x>0$. Furthermore, through the previously mentioned results, we can prove the global-in-$x$ $C^\infty$ regularity up to the boundary $y = 0$ and obtain the asymptotic behavior of the Prandtl equation when $x\rightarrow \infty$.
    \end{remark}
	
	We now comment on the proof of Theorem \ref{prandtl}. 
	
	First, we use a generalization of the Crocco transformation:
	\begin{equation}\label{crocco}
		\xi=x,\quad \eta=\frac{u(x,y)}{U(x)},
	\end{equation}
	we obtain for $\omega(\xi ,\eta)=\frac{u_y(x,y)}{x^{\frac{m-1}{2}}U(x)}$ the following equation:
	\begin{equation}\label{qy}
		L(\omega):=\nu \omega^2\omega_{\eta\eta}-\eta A\omega_{\xi}+(\eta^2-1)B\omega_{\eta}-\eta C\omega=0,
	\end{equation}
	in the domain $\varOmega=\{0<\xi<X,0<\eta<1\}$, with the boundary conditions
	\begin{equation}\label{bd3}
		\omega|_{\eta=1}=0,\quad (\nu \omega\omega_{\eta}-v_1\omega +B)|_{\eta=0}=0,
	\end{equation}
	where 
	\begin{equation}
		A=\xi V(\xi),\quad B=mV(\xi)+\xi V_{\xi}(\xi),\quad C=\frac{3m-1}{2}V(\xi)+\xi V_{\xi}(\xi).
	\end{equation}

    Note that the equation (\ref{qy}) is a degenerate parabolic equation. When $\xi\rightarrow 0$, let $Y(\eta)=\omega(0,\eta)$, it degenerates to the following elliptic equation:
    \begin{equation}\label{begin}
    	L(Y):=\nu Y^2Y_{\eta\eta}+(\eta^2-1)ma Y_{\eta}-\eta \frac{3m-1}{2}aY=0,\quad 0<\eta <1,
    \end{equation} 
    and the boundary conditions become
    \begin{equation}\label{bd4}
    	Y(1)=0,\quad (\nu YY_{\eta}+ma)|_{\eta=0}=0.
    \end{equation}
    
    Second, the key point is to establish the existence and prior estimations of the solution to the problem (\ref{begin}), (\ref{bd4}). Oleinik and Samokhin solve the above difficulties through the maximum principle of the elliptic equation for $m\geq 1$. This method cannot extend to the problem for $m<\frac{1}{3}$. Our basic idea is to establish equivalence between a solution to the problem (\ref{begin}), (\ref{bd4}) and the self-similar solution, and to use the properties of self-similar solutions to obtain the properties of this solution.
    
    Next, we note the solution $Y(\eta)$ to the equation (\ref{begin}) is an approximation of the solution $\omega(\xi,\eta)$ to the Prandtl equation (\ref{ps}) when $\xi\rightarrow 0$. And we establish the existence of solutions for problem (\ref{qy}), (\ref{bd3}) by the line method: replacing $\omega_{\xi}$ with the difference quotient, obtain a series of ordinary differential equations, prove the existence theorem and, passing to the limit, establish the existence of a solution to the original equation. Furthermore, we obtain prior estimates of $\omega(\xi,\eta)$ based on the properties of $Y(\eta)$.
    
    Finally, we use the inverse change of variables (\ref{crocco}) to pass from the solution $\omega(\xi,\eta)$ of problem (\ref{qy}), (\ref{bd3}) to the solutions of problem (\ref{ps}), (\ref{bd}). And then we obtain Theorem \ref{prandtl} with the help of the properties of the solution $\omega(\xi,\eta)$.

    We finish this introduction by outlining the rest of the paper. First, in Section II, we establish the existence and uniqueness of the solution to the problem (\ref{begin}), (\ref{bd4}) with the help of the self-similar solution. Based on the results in section II, we use the line method to obtain the existence and prior estimates of the solutions for problem (\ref{qy}), (\ref{bd3}) and then establish the local well-posedness of solutions for the Prandtl system (\ref{ps}), (\ref{bd}) in Section III. Finally, we similarly obtain the well-posedness results to three-dimensional flow around a cone similar to the two-dimensional symmetry case in Section IV.

	\section{Self-similar solutions}
	
	In the analysis of differential equations describing actual physical processes, it is often useful to examine some special solutions, which allow us to make conclusions about the main properties of the phenomena in equation. The simplest case of special solutions is that of a problem with the so-called property of self-similarity. A characteristic feature of self-similar problems of the boundary layer theory is that velocity component profiles for the flow in this case form a family of homothetic curves.
	
	Assume $U(x)$ and $v_0(x)$ satisfy the condition (\ref{assum}), consider $u(x,y)$ has the form
	\begin{equation}
		u(x,y)=x^mf^{'}_z(z,x),
	\end{equation}
	where $z=yx^{\frac{m-1}{2}}$. And we can obtain the following differential equation:
	\begin{equation}\label{omega}
		\nu f^{'''}_z+\left(\frac{m+1}{2}V+xV_x\right)ff^{''}_Z+(mV+xV_x)(1-f^{'2}_z)-xV \left(f^{'}_z\frac{\partial f^{'}_z}{\partial x}-\frac{\partial (f-f(0,x))}{\partial x}f^{''}_z\right)=0
	\end{equation}
	where $f_0\equiv f(0,x)=-\frac{v_1}{\frac{1+m}{2}V+xV_x}$, with the boundary conditions
	\begin{equation}
		f(0,x)=f_0,\quad f^{'}_z(0,x)=0,\quad f^{'}_z(z,x)\rightarrow 1,\quad \text{as}\quad z\rightarrow\infty.
	\end{equation}
	
	Introduce new variable $\xi=x, \eta=f^{'}(z)$ in (\ref{omega}), we obtain the following equation for $\omega(\xi,\eta)=f^{''}(z)$:
	\begin{equation}
		\nu \omega^2\omega_{\eta\eta}-\eta \xi V\omega_{\xi}+(\eta^2-1)(mV+\xi V_\xi)\omega_{\eta}-\eta \left(\frac{3m-1}{2}V+\xi V_\xi\right)\omega=0,
	\end{equation}
	with the boundary conditions
	\begin{equation}
		\omega|_{\eta=1}=0,\quad \left.\left(\nu \omega\omega_{\eta}-v_1\omega +(mV+\xi V_\xi)\right)\right|_{\eta=0}=0,
	\end{equation}
	This is exactly the equation (\ref{qy}), (\ref{bd3}) mentioned earlier.
	
	When $x\rightarrow 0$, we can obtain $u(x,y)$ has the following self-similarity form:
	\begin{equation}
		u(x,y)=x^mf^{'}(z),
	\end{equation}
	where $z=yx^{\frac{m-1}{2}}$ and $f(z)$ satisfies the following equation:
	\begin{equation}\label{si}
		f^{'''}+\frac{m+1}{2}aff^{''}+ma(1-f^{'2})=0,
	\end{equation}
	with the boundary conditions
	\begin{equation}
		f(0)=0,\quad f^{'}(0)=0, \quad f^{'}(z)\rightarrow 1 \quad \text{as} \quad z\rightarrow\infty.
	\end{equation}
	
	Introduce new variable $\eta=f^{'}(z)$ in (\ref{si}), we obtain the following equation for $Y(\eta)=f^{''}(z)$:
	\begin{equation}
		\nu Y^2Y_{\eta\eta}+(\eta^2-1)ma Y_{\eta}-\eta \frac{3m-1}{2}aY=0,\quad 0<\eta <1,
	\end{equation} 
	and the boundary conditions become
	\begin{equation}
		Y(1)=0,\quad (\nu YY_{\eta}+ma)|_{\eta=0}=0.
	\end{equation}
	This is exactly the equation (\ref{begin}), (\ref{bd4}) mentioned earlier. It is easy to see that the solution $Y(\eta)$ is an approximation of the solution $\omega(\xi,\eta)$ when $\xi\rightarrow 0$, since $u(x,y)\sim x^mf^{'}(yx^{\frac{m-1}{2}})$ as $x\rightarrow 0$. 
	
	Consider the following Falkner-Skan equation\cite{FS}:
	\begin{equation}\label{fs}
		f^{'''}+ff^{''}+\beta(1-f^{'2})=0,
	\end{equation}
    with the boundary conditions
    \begin{equation}\label{bd2}
	    f(0)=f_0,\quad f^{'}(0)=0,\quad f^{'}(z)\rightarrow 1 \quad\text{as}\quad z\rightarrow \infty.
    \end{equation}
    
	The special case $\beta=0$ of equation (\ref{fs}), in which the wedge reduces to a flat plate, was solved numerically by Blasius (1908 \cite{B}). And the case $\beta=\frac{1}{2}$, which describes the three-dimensional flow of a viscous fluid against a plane wall or the planar flow of a viscous fluid against a wedge with $\varphi=\frac{\pi}{2}$, was solved by Homann(1936 \cite{Ho}). Weyl(1942 \cite{W}) gave rigorous proof of the solubility of the boundary-value problem (\ref{fs}) and (\ref{bd2}) for any non-negative constant $\beta$. He showed that the problem has a solution $f(z)$ whose first derivative increases with $z$ and whose second derivative decreases and tends to zero as $z \rightarrow \infty$. Coppel(1959 \cite{C}) simplified Weyl's approach to studying the equation from the standpoint of the theory of differential equations. Schlichting(1968 \cite{S}) and Hartman(1964 \cite{H1}, 1972 \cite{H2}) have investigated the boundary value problem (\ref{fs}), (\ref{bd2}) by using numerical and analytical methods. Wang, Gao, and Zhang (1998 \cite{G}) presented a new approach to studying the problem (\ref{fs}), (\ref{bd2}) and provided some new information about a normal solution to the problem.

	Next, following Hartman\cite{H1}, Ch.XIV, Part II, we list several results concerning the existence, uniqueness, and asymptotic behavior of solutions to certain boundary conditions for equation (\ref{fs}). 
	
	Consider equation (\ref{fs}) with the boundary conditions
	\begin{equation}\label{bd5}
		f(0)=f_0, \quad f^{'}(0)=f_1, \quad f^{'}(z) \rightarrow 1 \text { as } z \rightarrow \infty,
	\end{equation}
	where $f_0, f_1= \text{const}$. On solutions of problem (\ref{fs}), (\ref{bd5}) we impose the following additional restriction:
	\begin{equation}\label{add}
		0<f^{'}(z)<1 \quad \text{for} \quad 0\leq z<\infty.
	\end{equation}
	
	\begin{theorem}\label{exi}
		Let $\beta>0,-\infty<f_0<+\infty, 0 \leq f_1<1$. Then there exists one and only one solution of problem (\ref{fs}), (\ref{bd5}), (\ref{add}) such that
		\begin{equation}
			f^{''}(z)>0 \quad \text { for } \quad 0 \leq z<\infty.
		\end{equation}
	\end{theorem}
	
	\begin{theorem}\label{asy}
		Let $\beta \geq 0$, and let $f(z)$ be a solution of problem (\ref{fs}), (\ref{bd5}), (\ref{add}). Then there exist constants $c_1>0$ and $c_2$ such that the following asymptotic formulas hold as $z \rightarrow \infty$ :
		\begin{equation}
			1-f^{'} \sim c_1z^{-1-2\beta}exp(-\frac{z^2}{2}-c_2z),\quad f^{''}\sim z(1-f^{'}).
		\end{equation}
	\end{theorem}

	Obviously, we can obtain the existence, uniqueness and asymptotic behavior of solutions for equation (\ref{si}) similar to Theorem \ref{exi} and Theorem \ref{asy}. Furthermore, we obtain these properties of solutions for equation (\ref{begin}),(\ref{bd4}). 
	
	\begin{theorem}\label{Y}
		Problem (\ref{begin}),(\ref{bd4}) has one and only one solution with the following properties:
		
		\begin{equation}\label{Y_0}
			\begin{gathered}
				M_5(1-\eta)\sigma \leq Y(\eta) \leq M_6(1-\eta)\sigma,\\
				M_6(1-\eta)(\sigma-K)\leq Y,\quad\text{for} \quad 0<\eta_0\leq\eta<1, \\
				-M_7\sigma \leq Y_{\eta} \leq -M_8\sigma,\\
				-M_9\leq  YY_{\eta\eta} \leq -M_{10},
			\end{gathered}
		\end{equation}
		where $M_i,K=\text{const}>0,i=5,\cdots,10$, $\sigma(\eta)=\sqrt{-\ln \mu (1-\eta)}$,  $\mu=\text{const}$, $0<\mu<1$, $\sigma>K$ for $\eta>\eta_0>0$.
	\end{theorem} 
	
	\begin{proof}
		The solution of equation (\ref{si}) has the following properties:
		\begin{equation}
			\begin{gathered}
				f(0)\geq 0,\quad f(z)>0, \quad \text{for}\quad 0< z<\infty,\\
				0<f^{'}(z)<1,\quad \text{for}\quad 0<z<\infty,\\
				f^{''}(z)>0,\quad \text{for}\quad 0\leq z<\infty,\\
				f^{'''}(z)<0,\quad \text{for}\quad 0< z<\infty,\\
				1-f^{'}(z)\sim C_1z^{-1-2\beta}exp(-\frac{z^2}{2}-C_2z),\quad \text{as}\quad z\rightarrow \infty,
			\end{gathered}
		\end{equation}
	    where $C_1,C_2=\text{const}$, $C_1>0$. We can easily obtain $Y(\eta)$ is bounded from below and above for $0<\eta<1-\delta$, where $\delta$ is sufficiently small. And we have
		\begin{equation}
			1-f^{'}(z) \sim C_1z^{-1-2\beta}exp(-\frac{z^2}{2}-C_2z) \sim C_3exp(-C_4z^2),\quad \text{as}\quad z\rightarrow \infty,
		\end{equation}
	    where $C_3,C_4=\text{const}>0$. It is easy to obtain
	    \begin{equation}
	   	    Y(\eta) \sim C_5(1-\eta)\sigma(\eta),\quad \text{as}\quad \eta\rightarrow 1.
	    \end{equation}
		Thus, choosing suitable $M_5, M_6, K$, we complete the proof of the first two inequalities. 
		
		By the method in \cite{G}, it is easy to see that a solution for equation (\ref{begin}) if and only if it solves the integral equation
		\begin{equation}
			\nu Y(\eta)=\int_{\eta}^{1}\frac{(1-s)(m+ms+\frac{m+1}{2}s)a}{Y(s)}ds+(1-\eta)\frac{m+1}{2}a\int_{0}^{\eta}\frac{s}{Y(s)}ds.
		\end{equation}
		From this we know that 
		\begin{equation}
			\nu Y_{\eta}(\eta)=-ma\frac{1-\eta^2}{Y(\eta)}-\frac{m+1}{2}a\int_{0}^{\eta}\frac{s}{Y(s)}ds.
		\end{equation}
		
		We can obtain the following properties:
		\begin{equation}
			\begin{aligned}
				\nu Y_{\eta}&\geq -ma\frac{1-\eta^2}{M_5(1-\eta)\sigma}-\frac{m+1}{2}a\int_{0}^{\eta}\frac{s}{M_5(1-s)\sigma(s)}ds\\
				&\geq -ma\frac{1+\eta}{M_5\sigma}+\frac{m+1}{2}a(\frac{\eta}{M_5\sigma}-\frac{2(\sigma-\sigma(0))}{M_5})\\
				&\geq -M_7\sigma,\\
				\nu Y_{\eta}&\leq-ma\frac{1-\eta^2}{M_6(1-\eta)\sigma}-\frac{m+1}{2}a\int_{0}^{\eta}\frac{s}{M_6(1-s)\sigma(s)}ds\\
				&= -ma\frac{1+\eta}{M_6\sigma}+\frac{m+1}{2}a\left(\int_{0}^{\eta}\frac{ds}{M_6\sigma(s)}-\frac{2(\sigma-\sigma(0))}{M_6}\right)\\
				&\leq -M_8\sigma.
			\end{aligned}
		\end{equation}
		
		From equation (\ref{begin}) we find that
		\begin{equation}
			\nu YY_{\eta\eta}=(1-\eta^2)ma\frac{Y_{\eta}}{Y}+\eta\frac{3m-1}{2}.
		\end{equation}
		We have
		\begin{equation}
				|\nu YY_{\eta\eta}|\leq M_9,
		\end{equation}
	    And for $m\leq\frac{1}{3}$, we have
	    \begin{equation}
	    	\nu YY_{\eta\eta}\leq -(1+\eta)ma\frac{M_8}{M_6}+\eta\frac{3m-1}{2}\leq -M_{10}.
	    \end{equation}
    
        Now, let prove the estimates for $m>\frac{1}{3}$. It follows from (\ref{begin}) that
        \begin{equation}
        	\nu YY_{\eta \eta}=\left(1-\eta^2\right) ma \frac{Y_\eta}{Y}+\eta\frac{3m-1}{2}a.
        \end{equation}
        Set $R=YY_{\eta\eta}$. Differentiating equation (\ref{begin}) with respect to $\eta$，we obtain the following equation for $R$:
        \begin{equation}
        	Q(R):=\nu YR_\eta+\left[\nu Y_\eta+(\eta^2-1)\frac{ma}{Y}-\frac{\nu Y}{\eta}\right]R=-\frac{maY_\eta}{\eta}+\eta\frac{m-1}{2}aY_\eta.
        \end{equation}
    
        Let us consider $R+M_{10}$. We have
        \begin{equation}
        	Q(R+M_{10})=-\frac{maY_\eta}{\eta}+\eta\frac{m-1}{2}aY_\eta+\left[\nu Y_\eta+(\eta^2-1)\frac{ma}{Y}-\frac{\nu Y}{\eta}\right]M_{10}>0,
        \end{equation}
        provided that $M_{10}$ is small enough.
        
        It follows from the second inequality for $Y$ that there is a sequence $\eta_n\rightarrow 1$, $n\rightarrow \infty$, such that
        \begin{equation}
        	\begin{aligned}
        		Y_\eta|_{\eta=\eta_n}&\leq M_6\left.\left(-\sigma+K+\frac{1}{2\sigma}\right)\right|_{\eta=\eta_n},\\
        		R|_{\eta=\eta_n}&= \frac{1}{\nu}\left.\left((1-\eta^2) ma \frac{Y_\eta}{Y}+\eta\frac{3m-1}{2}a\right)\right|_{\eta=\eta_n}\\
        		&\leq\frac{1}{\nu}\left.\left[\frac{ma(1+\eta)}{\sigma}\left(-\sigma+K+\frac{1}{2\sigma}+\eta\frac{3m-1}{2}a\right)\right]\right|_{\eta=\eta_n}\\
        		&< -M_{10},
        	\end{aligned}
        \end{equation} 
        if $n$ is sufficiently large.
        
        And we have
        \begin{equation}
        	R(0)=\frac{maY_\eta(0)}{\nu Y(0)}<-M_{10}.
        \end{equation}
        Therefore, $R+M_{10}<0$ for $0\leq\eta\leq \eta_n$. Since $\eta_n\rightarrow 1$ as $n\rightarrow \infty$, we have $R<-M_{10}$ for $0\leq\eta<1$.
        
	\end{proof}

	\section{Symmetric stationary boundary}
	
	The proof of the existence of solutions for problem (\ref{qy}), (\ref{bd3}) is accomplished by the line method which amounts to the construction of approximate solutions for the given problem by means of solutions of a ordinary differential equations.
		
	Set $f^k\equiv f^k(\eta)\equiv f(kh,\eta)$, where $h=\text{const}>0,k=0,1,2,\cdots,[X/h]$. We replace (\ref{qy}), (\ref{bd3}) with the system of differential equations:
	\begin{equation}\label{line}
		\begin{aligned}
			L_k(\omega):=\nu (\omega^k)^2\omega^k_{\eta\eta}-\eta (A^k+\mu_k h)\frac{\omega^k-\omega^{k-1}}{h}+(\eta^2-1)B^k\omega^k_{\eta}-\eta C^k\omega^k=0,\\
			\quad 0\leq \eta<1,\quad k=0,1,\cdots,[X/h],
		\end{aligned}
	\end{equation}
	with the conditions
	\begin{equation}\label{bd6}
		\omega^k(1)=0,\quad (\nu \omega^k \omega^k_{\eta}-v_1^k\omega^k+B^k)|_{\eta=0}=0.
	\end{equation}
	
	Here $\mu_k=0$ for $k=0$, and $\mu_k$, for $k\geq 1$, is a sufficiently large constant to be chosen later and $\mu_k\geq \mu_{k-1}$. It should be noted that when $m<1$, $\mu_k\equiv 0$ for $k\geq 0$. Obviously, the system (\ref{line}), (\ref{bd6}) is just the system (\ref{begin}) and (\ref{bd4}) when $k=0$. Next, we use the solution of problem (\ref{begin}),(\ref{bd4}) to proof the existence of solutions for problem (\ref{line}),(\ref{bd6}).
	
	In what follows, $M_i$, $K_i$, $C_i$, $N_i$ stand for positive constants independent of $h$.
	
	\begin{lemma}
		Problem (\ref{line}), (\ref{bd6}) admits a solution $\omega^k(\eta), k=1,2,\cdots,[X/h]$, which is continuous for $0\leq \eta \leq 1$ and infinitely differentiable on the segment $0\leq \eta<1$, provided that $U, v_0$ satisfy the conditions (\ref{assum}). The following estimate holds for this solution 
		\begin{equation}\label{VV}
			K_1(1-\eta)\leq \omega^k(\eta) \leq K_2(1-\eta)\sigma
		\end{equation}
		for $kh\leq X, h\leq h_0, h_0=\text{const}>0$, where $\sigma(\eta)=\sqrt{-\ln \mu(1-\eta)}$, $\mu=\text{const}$, $0<\mu<1$.
	\end{lemma}

	\begin{proof}
		We obtain the solution of problem (\ref{line}), (\ref{bd6}) as a limit ($\varepsilon \rightarrow 0$) of solutions of the following system
		\begin{equation}\label{epsilon}
			\begin{aligned}
				L_{\varepsilon,k}(\omega):=(\nu(\omega^k)^2+\varepsilon)\omega^k_{\eta\eta}-\eta (A^k+\mu_k h)\frac{\omega^k-\omega^{k-1}}{h}+(\eta^2-1)B^k\omega^k_{\eta}-\eta C^k\omega^k=0,\\
				0\leq \eta<1,\quad k=1,\cdots,[X/h],\quad \varepsilon>0,
			\end{aligned}
		\end{equation}
		supplemented with the conditions (\ref{bd6}). We have established these result for $k=0$ in Theorem \ref{Y}. we obtain a prior estimate from below by induction. Setting $H^k=K_1(1-\eta)$, we find that 
		\begin{equation}
			L_{\varepsilon,k}(H)\geq K_1(1-\eta)\left[(1+\eta)B^k-\eta C^k\right]>0
		\end{equation}
		for $0<\eta<1$ and $kh\leq X$. Setting $l_{\varepsilon,k}(\omega):=\left.\left(\nu \omega_{\eta}-v_1^k+\frac{B^k}{\omega^k}\right)\right|_{\eta=0}$, we obtain
		\begin{equation}
			l_{\varepsilon,k}(H)=-\nu K_1-v^k_1+\frac{B^k}{K_1}>0,
		\end{equation}
		provided that $K_1$ is sufficiently small. Set $S^k=\omega^k-H^k$. Then we have
		\begin{equation}
			\begin{gathered}
				(\nu(\omega^k)^2+\varepsilon)S^k_{\eta\eta}-\eta (A^k+\mu_k h)\frac{S^k-S^{k-1}}{h}+(\eta^2-1)B^kS^k_{\eta}-\eta C^kS^k<0,\\
				\left.\left(\nu S^k_{\eta}-\frac{B^k}{H^k\omega^k}S^k\right)\right|_{\eta=0}<0.
			\end{gathered}
		\end{equation}
		This inequalities imply that $S^k\geq 0$ and $\omega^k\geq K_1(1-\eta)$, since $-\eta \frac{(A^k+\mu_k h)}{h}-\eta C^k\leq 0$ and $-\frac{B^k}{H^k\omega^k}<0$.
		
		Consider the system (\ref{epsilon}) with the boundary conditions:
		\begin{equation}\label{xbd}
			\omega^k(1)=0,\quad \left.\left(\nu \omega^k_{\eta}-v_1^k+\frac{B^k}{\psi(\omega^k)}\right)\right|_{\eta=0}=0.
		\end{equation}
		where $\psi(x)$ is an infinitely differentiable function defined for $-\infty <x<\infty$ and such that $\psi(x)=x$ for $x\geq K_1$, $\psi(x)=K_1/2$ for $x\leq K_1/4$ and $0\leq \psi'(x)\leq 1$ for $K_1/4\leq x\leq K_1$. 
		
		For a solution $\tilde{\omega}^k$ of problem (\ref{epsilon}), (\ref{xbd}), we have $\tilde{\omega}^k\geq H^k$. Indeed, let $\tilde{S}^k=\tilde{\omega}^k-H^k$. Then
		\begin{equation}
			\left.\left(\nu H^k_{\eta}-v_1^k+\frac{B^k}{\psi(H^k)}\right)\right|_{\eta=0}=-\nu K_1-v^k_1+\frac{B^k}{K_1}>0,
		\end{equation}
		because of the choice of $K_3$, and
		\begin{equation}
			\left.\left(\nu \tilde{S}^k_{\eta}-\frac{B^k\psi'(\chi_k)}{\psi(H^k)\Psi(\tilde{\omega}^k)}\tilde{S}^k\right)\right|_{\eta=0}<0,\quad \psi'(\chi_k)\geq 0.
		\end{equation}
		
		Hence, $\psi(\tilde{\omega}^k)=\tilde{\omega}^k$, and the solution $\tilde{\omega}^k$ is also a solution of problem (\ref{epsilon}), (\ref{bd6}).
		
		The existence of solutions for problem (\ref{epsilon}), (\ref{xbd}) for $\varepsilon>0$ can be established by the well-known method based on the Leray-Schauder theorem.
		
		\begin{theorem}{(Leray-Schauder)}\label{LS}
			In a Banach space $X$, consider a family of mappings $y=T(x, k)$,
			where $x, y \in X $, $k$ is a real parameter varying on the segment $a \leq k \leq b$. 
			
			Assume that:
			
			1) $T(x, k)$ is defined for all $x \in X$ and $a \leq k \leq b$;
			
			2) for any fixed $k$, the operator $T(x, k)$ is continuous on $X$, i.e., for any $x^0 \in X$ and any $\varepsilon>0$, there is a constant $\delta>0$ such that $\left\|T(x, k)-T\left(x^0, k\right)\right\|<\varepsilon$, provided that $\left\|x-x^0\right\|<\delta$;
			
			3) on bounded subsets of $X$, the operators $T(x, k)$ are uniformly continuous with respect to $k$, i.e., for any bounded set $X_0 \subset X$ and any $\varepsilon>0$, there is a constant $\delta>0$ such that for $x \in X_0$ and $|k_1-k_2|<\delta$, $k_1, k_2 \in[a, b]$ we have $\left\|T\left(x, k_1\right)-T\left(x, k_2\right)\right\|<\varepsilon$;
			
			4) for any fixed $k, T(x, k)$ is a compact operator, i.e, it maps each bounded subset of $X$ into a compact subset of $X$;
			
			5) there exists a constant $M$ such that for any solution $x$ of the equation $x-T(x, k)=0(x \in X, k \in[a, b])$, we have $\|x\| \leq M$;
			
			6) the equation $x-T(x, a)=0$ has a unique solution in $X$.
			
			Then the equation $x-T(x, b)=0$ admits a solution in $X$.
		\end{theorem}
		
		Consider the following system of differential equations depending on a parameter $\gamma\in [0,1]$:
		\begin{equation}\label{omega1}
			(\nu\gamma(\omega^k)^2+\varepsilon)\omega^k_{\eta\eta}-\eta (A^k+\mu_k h)\frac{\omega^k-\omega^{k-1}}{h}+(\eta^2-1)B^k\omega^k_{\eta}-\eta C^k\omega^k=0,
		\end{equation}
		together with the boundary conditions
		\begin{equation}\label{omega2}
			\omega^k(1)=0,\quad \left.\left(\nu \omega^k_{\eta}-v_1^k+B^k\left[\varphi(\gamma \omega^k)\omega^k+\frac{1}{\psi(0)}\right]\right)\right|_{\eta=0}=0,
		\end{equation}
		where
		\begin{equation}
			x\varphi(x)=\frac{1}{\psi(x)}-\frac{1}{\Psi(0)}=-x\int_{0}^{1}\frac{\psi'(sx)}{\psi^2(sx)ds},\quad \varphi(x)\leq 0.
		\end{equation}
		
		For $\gamma=0$, problem (\ref{omega1}), (\ref{omega2}) is linear, and it turns into problem (\ref{epsilon}), (\ref{xbd}) for $\gamma=1$.
		
		Let us verify the conditions of the Leray-Schauder theorem \ref{LS} for this system. Consider the operator $T(\theta, \gamma)=\omega$ which maps any vector valued function $\theta$ of class $C^\alpha([0,1])$ into $\omega \equiv(\omega^1, \cdots, \omega^m), m=[X / h]$, where $\omega$ is a solution of the following linear system of differential equations:
		\begin{equation}\label{theta1}
			(\nu\gamma(\theta^k)^2+\varepsilon)\omega^k_{\eta\eta}-\eta (A^k+\mu_k h)\frac{\omega^k-\omega^{k-1}}{h}+(\eta^2-1)B^k\omega^k_{\eta}-\eta C^k\omega^k=0,
		\end{equation}
		with the boundary conditions
		\begin{equation}\label{theta2}
			\omega^k(1)=0,\quad \left.\left(\nu \omega^k_{\eta}-v_1^k+B^k\left[\varphi(\gamma \theta^k)\omega^k+\frac{1}{\psi(0)}\right]\right)\right|_{\eta=0}=0.
		\end{equation}
		
		For $\gamma=0$, problem (\ref{theta1}), (\ref{theta2}) admits a unique solution, which follows from the fact that the problem is linear, the coefficient of $\omega^k$ in (\ref{theta1}) is negative, and $\varphi(x) \leq 0$ for $kh\leq X$.
		
		The solutions $\omega^k$ of the problem (\ref{omega1}), (\ref{omega2}) are uniformly (with respect to $\gamma$) bounded, together with their second derivatives. We have
		\begin{equation}
			\left.\left(\nu H^k_{\eta}-v_1^k+B^k\left[\varphi(\gamma \omega^k)H^k+\frac{1}{\psi(0)}\right]\right)\right|_{\eta=0}=-\nu K_1-v_1^k+B^k\left[\varphi(\gamma \omega^k)K_1+\frac{2}{K_1}\right]>0,
		\end{equation}
		provided that $K_1$ is small enough and is independent of $\gamma, h, \varepsilon$. Hence, for $S^k=\omega^k-H^k$, we obtain
		\begin{equation}
			\begin{gathered}
				(\nu\gamma(\omega^k)^2+\varepsilon)S^k_{\eta\eta}-\eta (A^k+\mu_k h)\frac{S^k-S^{k-1}}{h}+(\eta^2-1)B^kS^k_{\eta}-\eta C^kS^k<0,\\
				\left.\left(\nu S^k_{\eta}+B^k\varphi(\gamma \omega^k)S^k\right)\right|_{\eta=0}<0,
			\end{gathered}
		\end{equation}
		which readily imply that $S^k\geq 0$ and $\omega^k\geq K_1(1-\eta)$ for all $\gamma$, $kh\leq X, 0\leq \eta\leq 1, 0<\varepsilon\leq 1$.
		
		Let $\omega^k=(K_3-e^{\alpha\eta})\bar{\omega}^k$. Choosing $\alpha, K_3$ suitably large, we obtain
		\begin{equation}
			(\nu\gamma(\omega^k)^2+\varepsilon)\bar{\omega}^k_{\eta\eta}-\eta (A^k+\mu_k h)\frac{\bar{\omega}^k-\bar{\omega}^{k-1}}{h}+\bar{B}^k\bar{\omega}^k_{\eta}+\bar{C}^k\bar{\omega}^k=0,
		\end{equation} 
		where $\bar{C}^k<0$, and the boundary conditions
		\begin{equation}
			\bar{\omega}^k(1)=0,\quad \left.\left(\nu \bar{\omega}^k_{\eta}+\bar{D}^k\bar{\omega}^k+\bar{E}^k\right)\right|_{\eta=0}=0,
		\end{equation}
		where $\bar{D}^k<0$. Indeed, we have
		\begin{equation}
			\begin{aligned}
				\bar{C}^k&=-\eta C^k-\frac{(\nu\gamma(\omega^k)^2+\varepsilon)\alpha^2e^{\alpha\eta}}{K_3-e^{\alpha\eta}}-\frac{(\eta^2-1)B^k\alpha e^{\alpha\eta}}{K_3-e^{\alpha\eta}},\\
				\bar{D}^k&=B^k\varphi(\gamma \omega^k)-\frac{\alpha e^{\alpha\eta}}{K_3-e^{\alpha\eta}}.
			\end{aligned}	
		\end{equation}
		By the maximum principle, it easily follows that $\omega^k$ is bounded uniformly in $\gamma$.
		
		The estimate for the derivatives $\omega^k_{\eta}$, uniform in $\gamma$, follows from the first order equations obtained from (\ref{omega1}) for the function $\omega^k_{\eta}$, as well as from the estimate for $\omega^k_{\eta}$ at $\eta=0$ obtained from (\ref{omega2}). Indeed, it follows from (\ref{omega1}), for $z^k=\omega^k_{\eta}$, that
		\begin{equation}
			(\nu\gamma(\omega^k)^2+\varepsilon)z^k_{\eta}+(\eta^2-1)B^kz^k=\eta C^k\omega^k+\eta (A^k+\mu_k h)\frac{\omega^k-\omega^{k-1}}{h}.
		\end{equation}
		Let 
		\begin{equation}
			I(\eta)=\frac{(\eta^2-1)B^k}{\nu\gamma(\omega^k)^2+\varepsilon}, J(\eta)=\frac{\eta C^k\omega^k+\eta (A^k+\mu_k h)\frac{\omega^k-\omega^{k-1}}{h}}{\nu\gamma(\omega^k)^2+\varepsilon},
		\end{equation}
		and we can obtain
		\begin{equation}
			z^k_{\eta}+Iz^k=J,
		\end{equation}
		we have
		\begin{equation}
			z^k(\eta)= \left(z^k(0)+\int_{0}^{\eta}J(\tau)e^{\int_{0}^{\tau}I(s)ds}d\tau\right)e^{-\int_{0}^{\eta}I(s)ds}.
		\end{equation}
		Hence, $|\omega^k_{\eta}|\leq K_4(\varepsilon)$.
		
		The derivatives $\omega^k_{\eta \eta}$ are found from (\ref{omega1}), since we have
		\begin{equation}
			\varepsilon|\omega^k_{\eta\eta}| \leq \left|\eta (A^k+\mu_k h)\frac{\omega^k-\omega^{k-1}}{h}+(1-\eta^2)B^k\omega^k_{\eta}+\eta C^k\omega^k\right|\leq K_5(\varepsilon).
		\end{equation} 
		
		According to the well-known estimates of Schauder type (see Gilbarg and Trudinger \cite{D}), we can estimate (uniformly in $\gamma$) the $C^{2,\alpha}$ norm of solution $\omega^k$ to problem (\ref{theta1}), (\ref{theta2}). The constant in this estimates depends on the $C^\alpha$ norm of $\theta(\eta)$.
		
		It follows that the operator $T(\theta, \gamma)$ maps a bounded set formed by functions $\theta$ in $C^\alpha([0,1])$ into a compact set formed by $\omega$ in $C^2([0,1])$. The continuity properties of $T(\theta, \gamma)$ required by the Leray-Schauder theorem follow from the continuity of the conditions that hold for the difference of the solutions of problem (\ref{theta1}), (\ref{theta2}) for different $\theta$ and $\gamma$, as well as from the estimates, which hold uniformly with respect to $\gamma$, for these solutions and their derivatives.
		
		Thus, the existence of a solution for problem (\ref{epsilon}), (\ref{xbd}) with $\varepsilon>0$ is a consequence of the Leray-Schauder theorem. 
		
		Now, for these solutions, let us prove uniform estimates from above. Set $H_1=K_2(1-\eta)\sigma$. For $\mu$, we can take any constant such that $\mu<1$ and $\sigma>1$ for $0\leq \eta<1$. We have
		\begin{align*}
			L_{\varepsilon,k}(H_1)=&-\varepsilon K_2(\frac{1}{2(1-\eta)\sigma}+\frac{1}{4(1-\eta)\sigma^3})-\nu K_2(1-\eta)\sigma (\frac{K^2_2}{2}+\frac{K^2_2}{4\sigma})\\
			&+(\eta^2-1)B^kK_2(-\sigma+\frac{1}{2\sigma})-\eta C^kK_2(1-\eta)\sigma\\
			\leq& -K_2(1-\eta)\sigma \left[\nu\frac{K^2_2}{2}+\nu\frac{K^2_2}{4\sigma}-(1+\eta)(1-\frac{1}{2\sigma^2})B^k+\eta C^k\right]\\
			<& 0
		\end{align*}
		for $0< \eta < 1$, provided that $K_2$ is suitably large and independent of $h, \varepsilon$. Further, we have
		\begin{equation}
			l_{\varepsilon,k}(H_1)=\left.\left(\nu K_2(-\sigma+\frac{1}{2\sigma})-v^k_1+\frac{B^k}{K_2\sigma}\right)\right|_{\eta=0}<0,
		\end{equation}
		if $\mu<e^{-1/2}$ and $K_2$ is chosen sufficiently large. The difference $S^k_1=\omega^k-H^k_1$ satisfies the inequalities
		\begin{equation}
			\begin{gathered}
				(\nu(\omega^k)^2+\varepsilon)S^k_{1\eta\eta}-\eta (A^k+\mu_k h)\frac{S^k_1-S^{k-1}_1}{h}+(\eta^2-1)B^kS^k_{1\eta}-\eta C^kS^k_1+\nu H^k_{1\eta\eta}(\omega^k+H^k_1)S^k>0,\\
				\left.\left(\nu S^k_{1\eta}-\frac{B^k}{H^k_1\omega^k}S^k_1\right)\right|_{\eta=0}>0.
			\end{gathered}
		\end{equation}
		These inequalities readily imply that $S^k_1\leq 0$ and $\omega^k\leq K_2(1-\eta)\sigma$ for $0\leq\eta\leq 1$ and $kh\leq X$, since $-\eta \frac{(A^k+\mu_k h)}{h}-\eta C^k+\nu H^k_{1\eta\eta}(\omega^k+H^k_1)<0$ and $-\frac{B^k}{V^k_1\omega^k}<0$.
		
		From equations (\ref{epsilon}) and the boundary conditions (\ref{bd6}) for $\eta=0$, in combination with the two-sided estimates for $\omega^k$, we obtain estimates for $\omega_\eta^k, \omega_{\eta \eta}^k$ which hold on any segment $0 \leq \eta \leq 1-\delta, \delta=\text{const} >0$, uniformly with respect to $\varepsilon$. Differentiating equation (\ref{epsilon}) with respect to $\eta$, we can show that the derivatives of $\omega_k$ up to any given order are bounded on the segment $0 \leq \eta \leq 1-\delta$, uniformly with respect to $\varepsilon$.
		
		Consequently, from the family of solutions $w^k$ of problem (\ref{epsilon}), (\ref{bd6}) depending on the parameter $\varepsilon$, $0<\varepsilon \leq 1$, we can extract a sequence $\omega^k$ such that $\omega^k$, together with their derivatives of any given order, are uniformly convergent on the segment $0 \leq \eta \leq 1-\delta$, as $\varepsilon\rightarrow 0$. Obviously, the estimates (\ref{VV}) hold for the limit functions $\omega^k$. Further the limit function is continuous on $0 \leq \eta \leq 1$, it vanishes for $\eta=1$ and satisfies (\ref{line}), (\ref{bd6}).
		
	\end{proof}

    \begin{lemma}\label{wk}
    	Assume $U(x)$ and $v_0(x)$ satisfy the conditions (\ref{assum}). Then, if $kh\leq X$ and $X>0$ is sufficiently small, problem (\ref{line}),(\ref{bd6}) has a solution with the following properties:
    	\begin{equation}\label{ine}
    		\begin{gathered}
    			Y(1-M_{11}kh)\leq \omega^k \leq Y(1+M_{12}kh),\\
    			|\frac{\omega^k-\omega^{k-1}}{h}|\leq M_{13}Y, \\
    			Y_{\eta}(1+M_{14}kh)\leq \omega^k_{\eta} \leq Y_{\eta}(1-M_{15}kh),\quad \text{for}\quad m\geq \frac{1}{3},\\
    			-M_{16} \sigma\leq \omega^k_{\eta} \leq -M_{17} \sigma,\quad \text{for}\quad 0<m<\frac{1}{3},\\
    			-M_{18}\leq \omega^k \omega^k_{\eta\eta}\leq -M_{19},
    		\end{gathered}
    	\end{equation}
    	where $\sigma(\eta)=\sqrt{-\ln \mu(1-\eta)}$, $\mu=\text{const}$, $0<\mu<1$, $\delta_1<\delta$ , $\delta$ is sufficiently small.
    \end{lemma}
    
    \begin{proof}
    	The inequalities will be prove by induction. We have established these inequalities for $k=0$ in Theorem \ref{Y}. Let us show that there exist constants $M_{11},\cdots,M_{19}$ such that these inequalities being valid for $k-1$ imply their validity for $k,k\geq 1$, if $kh\leq X$ and $X$ is independent of $k$. Set
    	\begin{equation}
    		\Phi^k=Y(1-M_{11}kh).
    	\end{equation}
    	
    	Then
    	\begin{equation}
    		\begin{aligned}
    			L^k(\Phi)&=\nu Y^2Y_{\eta\eta}(1-M_{11}kh)^3+\eta (A^k+\mu_k h)M_{11}Y+(\eta^2-1)B^kY_{\eta}(1-M_{11}kh)\\
    			&\qquad -\eta C^kY(1-M_{11}kh)\\
    			&=L(Y)(1-M_{11}kh)+\nu Y^2Y_{\eta\eta}[(1-M_{11}kh)^2-1](1-M_{11}kh)\\
    			&\qquad  +\eta M_{11}Y(kh(a+a_1^k)+\mu_k h)+(\eta^2-1)Y_{\eta}(ma_1^k+kha_{1\xi}^k)(1-M_{11}kh)\\
    			&\qquad  -\eta Y(\frac{3m-1}{2}a_1^k+kha_{1\xi}^k)(1-M_{11}kh)\\
    			&\geq 2M_{11}M_{10}Ykh + \eta M_{11}Y(akh+\mu_k h)+(\eta^2-1)Y_{\eta}(ma_1^k+kha_{1\xi}^k)\\
    			&\qquad -\eta Y(\frac{3m-1}{2}a_1^k+kha_{1\xi}^k)-N_1(kh)^2\\
    			&>0,
    		\end{aligned}
    	\end{equation}
    	for $0<\eta<1$, $kh\leq X_1$, where $M_{11}$ is sufficiently large and independent of $h,k$.
    	
    	Setting $l_k(\omega):=\left.\left(\nu \omega_{\eta}-v_1^k+\frac{B^k}{\omega^k}\right)\right|_{\eta=0}$, we obtain
    	\begin{equation}
    		\begin{aligned}
    			l_k(\Phi)&=(\nu Y_{\eta}(1-M_{11}kh)-(b+b_1^k)+\frac{B^k}{Y(1-M_{11}kh)})|_{\eta=0}\\
    			&=-\nu M_{11}Y_{\eta}(0)kh-b_1^k+\frac{maM_{11}kh}{Y(0)(1-M_{11}kh)}+\frac{ma_1^k+kha_{1\xi}^k}{Y(0)(1-M_{11}kh)}\\
    			&>0,
    		\end{aligned}
    	\end{equation}
    	if $kh\leq X_1$ and $M_{11}$ is sufficiently large.
    	
    	Consider the functions $s^k=\omega^k-\Phi^k$. we have 
    	\begin{equation}\label{s1}
    		\nu(\omega^k)^2s^k_{\eta\eta}-\eta (A^k+\mu_k h)\frac{s^k-s^{k-1}}{h}+(\eta^2-1)B^ks^k_{\eta}-\eta C^ks^k+\Phi^k_{\eta\eta}(\Phi^k+\omega^k)s^k<0,
    	\end{equation}
    	for $0<\eta<1$, $kh\leq X_1$. We also have
    	\begin{equation}\label{s2}
    		s^k(1)=0,\quad \left.\left(\nu s^k_{\eta}-\frac{B^k}{\Phi^k\omega^k}s^k\right)\right|_{\eta=0}<0,
    	\end{equation}
    	for $kh\leq X_1$. And we have
    	\begin{equation}
    		\begin{gathered}
    			-\eta \frac{(A^k+\mu_k h)}{h}-\eta C^k+\Phi^k_{\eta\eta}(\Phi^k+\omega^k)
    			\leq -\eta ka-\eta\mu_k-\eta \frac{3m-1}{2}a+\eta N_2kh\leq 0,\\
    			-\frac{B^k}{\Phi^k\omega^k}<0.
    		\end{gathered}	
    	\end{equation}
    	
    	Since the coefficients of $s^k$ in (\ref{s1}) and (\ref{s2}) are negative for $kh\leq X_1$, we can obtain $s^k\geq 0$, i.e. $\omega^k\geq Y(1-M_{11}kh)$ for $0\leq\eta\leq1$ and $kh\leq X_1$. In a similar way we show that 
    	\begin{equation}
    		\omega^k\leq Y(1+M_{12}kh),
    	\end{equation}
    	for $0\leq\eta\leq1$ and $kh\leq X_2$, where $M_{12}$ is sufficiently large and independent of $h,k$.
    	
    	Set
    	\begin{equation}
    		\omega^k_{\eta}=z^k,\quad \frac{w^k-w^{k-1}}{h}=r^k.
    	\end{equation}
    	Let us write out the equations for $z^k$ and $r^k$, $k\geq 1$. Differentiating (\ref{line}) with respect to $\eta$, we get
    	\begin{equation}
    		\begin{aligned}
    			P_k(z):=&\nu (\omega^k)^2z^k_{\eta\eta}+(\eta^2-1)B^kz^k_{\eta}+\eta (2B^k-C^k)z^k-\eta (A^k+\mu_k h)\frac{z^k-z^{k-1}}{h}\\
    			&\quad +2\nu \omega^kz^kz^k_{\eta}-(A^k+\mu_k h)r^k-C^k\omega^k=0.
    		\end{aligned}
    	\end{equation}
    	
    	From the boundary condition (\ref{bd6}) we get
    	\begin{equation}
    		\left.\left(\nu z^k-v_1^k+\frac{B^k}{w^k}\right)\right|_{\eta=0}=0.
    	\end{equation}
    	
    	Subtracting from the equation (\ref{line}) for $\omega^k$ that for $\omega^{k-1}$ and dividing the result by $h$, we obtain
    	\begin{equation}
    		\begin{aligned}
    			&\nu (\omega^k)^2r^k_{\eta\eta}+(\eta^2-1)B^kr^k_{\eta}-\eta (A^{k-1}+\mu_{k-1}h)\frac{r^k-r^{k-1}}{h}-\eta C^kr^k-\eta \left(\frac{A^k-A^{k-1}}{h}+(\mu_k-\mu_{k-1})\right) r^k\\
    			&+\frac{\omega^k+\omega^{k-1}}{(\omega^{k-1})^2}\left(\eta (A^{k-1}+\mu_{k-1} h)r^{k-1}+(1-\eta^2)B^{k-1}z^{k-1}+\eta C^{k-1}\omega^{k-1}\right)r^k\\
    			&+(\eta^2-1)\frac{B^k-B^{k-1}}{h}z^{k-1}-\eta \frac{C^k-C^{k-1}}{h}\omega^{k-1}=0.
    		\end{aligned}
    	\end{equation}
    	
    	Similarly, from (\ref{bd6}) we obtain 
    	\begin{equation}
    		r^k(1)=0,\quad \left.\left(\nu r^k_{\eta}-\frac{B^k}{\omega^k\omega^{k-1}}r^k-\frac{v^k_1-v^{k-1}_1}{h}+\frac{B^k-B^{k-1}}{h\omega^{k-1}}\right)\right|_{\eta=0}=0.
    	\end{equation}
    	
    	Let us introduce the following functions:
    	\begin{equation}
    		\begin{aligned}
    			R_k(\varphi):=&\nu (\omega^k)^2\varphi^k_{\eta\eta}+(\eta^2-1)B^k\varphi^k_{\eta}-\eta (A^{k-1}+\mu_{k-1}h)\frac{\varphi^k-\varphi^{k-1}}{h}\\
    			&\quad-\eta C^k\varphi^k-\eta \left(\frac{A^k-A^{k-1}}{h}+(\mu_k-\mu_{k-1})\right)\varphi^k\\
    			&\quad+\frac{\omega^k+\omega^{k-1}}{(\omega^{k-1})^2}\left(\eta (A^{k-1}+\mu_{k-1} h)r^{k-1}+(1-\eta^2)B^{k-1}z^{k-1}+\eta C^{k-1}\omega^{k-1}\right)\varphi^k,\\
    			\rho_k(\varphi):=&\left.\left(\nu \varphi^k_{\eta}-\frac{B^k}{\omega^k\omega^{k-1}}\varphi^k\right)\right|_{\eta=0}.
    		\end{aligned}
    	\end{equation}
    	
    	Setting $\varphi^k=M_{13}Y$, we find that
    	\begin{equation}
    		\begin{aligned}
    			R_k(\varphi)=&M_{13}\{\nu (\omega^k)^2Y_{\eta\eta}+(\eta^2-1)B^kY_{\eta}-\eta C^kY-\eta \left(\frac{A^k-A^{k-1}}{h}+(\mu_k-\mu_{k-1})\right)Y \\
    			&\quad +\frac{\omega^k+\omega^{k-1}}{(\omega^{k-1})^2}\left(\eta (A^{k-1}+\mu_{k-1} h)r^{k-1}+(1-\eta^2)B^{k-1}z^{k-1}+\eta C^{k-1}\omega^{k-1}\right)Y\}\\
    			=&M_{13}\{L(Y)+\nu [(\omega^k)^2-Y^2]Y_{\eta\eta}+(\eta^2-1)(ma^k_1+kha^k_{1\xi})Y_{\eta}\\
    			&\quad-\eta (\frac{3m-1}{2}a^k_1+kha^k_{1\xi})Y-\eta \left(\frac{A^k-A^{k-1}}{h}+(\mu_k-\mu_{k-1})\right)Y\\
    			&\quad+\frac{\omega^k+\omega^{k-1}}{(\omega^{k-1})^2}\left(\eta (A^{k-1}+\mu_{k-1} h)r^{k-1}+(1-\eta^2)B^{k-1}z^{k-1}+\eta C^{k-1}\omega^{k-1}\right)Y \}.
    		\end{aligned}
    	\end{equation}
    	And we have 
    	\begin{equation}
    		R_k(\varphi)+\left|(\eta^2-1)\frac{B^k-B^{k-1}}{h}z^{k-1}-\eta \frac{C^k-C^{k-1}}{h}\omega^{k-1}\right|<0,
    	\end{equation}
    	for $0<\eta<1$, provided $kh\leq X_3$ and $M_{13}$ is sufficiently large and does not depend on $k,h$. This is possible, since
    	\begin{equation}
    		\begin{gathered}
    			M_{13}\frac{\omega^k+\omega^{k-1}}{(\omega^{k-1})^2}(1-\eta^2)B^{k-1}z^{k-1}Y+\left|(\eta^2-1)\frac{B^k-B^{k-1}}{h}z^{k-1}\right|\leq 0,\\
    			M_{13}\frac{\omega^k+\omega^{k-1}}{(\omega^{k-1})^2}\eta C^{k-1}\omega^{k-1}Y+\left|-\eta \frac{C^k-C^{k-1}}{h}\omega^{k-1}\right|\leq 0,\\
    			\left|\nu [(\omega^k)^2-Y^2]Y_{\eta\eta}+(\eta^2-1)(ma^k_1+kha^k_{1\xi})Y_{\eta}-\eta (\frac{3m-1}{2}a^k_1+kha^k_{1\xi})Y\right|\leq C_1Ykh,\\
    			\frac{\omega^k+\omega^{k-1}}{(\omega^{k-1})^2}\left(\eta (A^{k-1}+\mu_{k-1} h)r^{k-1}+(1-\eta^2)B^{k-1}z^{k-1}+\eta C^{k-1}\omega^{k-1}\right)Y\leq -C_2Y.
    		\end{gathered}
    	\end{equation}
    	
    	Let us calculate $\rho_k(\varphi)$. We have
    	\begin{equation}
    		\rho_k(\varphi)=\left. M_{13}\left(\nu Y_{\eta}-\frac{B^k}{\omega^k\omega^{k-1}}Y\right)\right|_{\eta=0}<-M_{13}\gamma,
    	\end{equation}
    	where $\gamma$ is a positive constant. It is easy to see that
    	\begin{equation}
    		\rho_k(\varphi)+\left.\left|-\frac{v^k_1-v^{k-1}_1}{h}+\frac{B^k-B^{k-1}}{h\omega^{k-1}}\right|\right|_{\eta=0}<0,
    	\end{equation}
    	if $kh\leq X_3$ and $M_{13}$ is sufficiently large and does not depend on $k,h$.
    	
    	Consider the functions $\phi^k_{\pm}=\varphi^k\pm r^k$. We have
    	\begin{equation}
    		\begin{aligned}\label{s3}
    			R_k(\phi_{\pm})=&\nu (\omega^k)^2\phi^k_{\pm\eta\eta}+(\eta^2-1)B^k\phi^k_{\pm\eta}-\eta A^{k-1}\frac{\phi^k_{\pm}-\phi^{k-1}_{\pm}}{h}\\
    			&+\left[(\omega^k+\omega^{k-1})\nu\omega^{k-1}_{\eta \eta}-\eta C^k-\eta \frac{A^k-A^{k-1}}{h}\right]\phi^k_{\pm}<0,\\
    		\end{aligned}
    	\end{equation}
    	for $0<\eta<1$, if $kh\leq X_3$. We also have
    	\begin{equation}\label{s4}
    		\rho_k(\phi_{\pm})=\left.\left(\nu \phi^k_{\pm\eta}-\frac{B^k}{\omega^k\omega^{k-1}}\phi^k_{\pm}\right)\right|_{\eta=0}<0,\quad \phi^k_{\pm}(1)=0,
    	\end{equation}
    	if $kh\leq X_3$. And we have
    	\begin{equation}
    		\begin{aligned}
    			&\frac{\omega^k+\omega^{k-1}}{(\omega^{k-1})^2}\left(\eta (A^{k-1}+\mu_{k-1} h)r^{k-1}+(1-\eta^2)B^{k-1}z^{k-1}+\eta C^{k-1}\omega^{k-1}\right)\\
    			&\quad-\eta C^k-\eta \left(\frac{A^k-A^{k-1}}{h}+(\mu_k-\mu_{k-1})\right)-\eta \frac{(A^{k-1}+\mu_{k-1}h)}{h}\\
    			\leq &-\eta \frac{3m-1}{2}a-\eta a-\eta (k-1)a-\eta \mu_k+\eta N_3kh\leq 0,
    		\end{aligned}
    	\end{equation}
    	for $kh\leq X_3$. Since the coefficients of $\phi^k_{\pm}$ in (\ref{s3}) and (\ref{s4}) are negative for $kh\leq X_3$, we can obtain $\phi^k_{\pm}\geq 0$, i.e. $\left|\frac{\omega^k-\omega^{k-1}}{h}\right|\leq M_{13}Y$ for $0\leq\eta\leq1$ and $kh\leq X_3$. 
    	
    	Let us estimate $z^k=w^k_{\eta}$ when $m\geq \frac{1}{3}$. Setting $F^k_1=Y_{\eta}(1+M_{14}kh)$, we find that 
    	\begin{equation}
    		\begin{aligned}
    			P_k(F_1)=&\nu (\omega^k)^2Y_{\eta\eta\eta}(1+M_{14}kh)+(\eta^2-1)B^kY_{\eta\eta}(1+M_{14}kh)+\eta(2B^k-C^k)Y_{\eta}(1+M_{14}kh)\\
    			&\quad-\eta (A^k+\mu_k h)Y_{\eta}M_{14}+2\omega^kY_{\eta}Y_{\eta\eta}(1+M_{14}kh)^2-(A^k+\mu_k h)r^k-C^k\omega^k\\
    			=&(1+M_{14}kh)[(L(Y))_{\eta}+((\omega^k)^2-Y^2)Y_{\eta\eta\eta}+(\eta^2-1)(ma^k_1+kha^k_{1\xi})Y_{\eta\eta}\\
    			&\quad+\eta(2ma^k_1+2kha^k_{1\xi}-\frac{3m-1}{2}a^k_1-kha^k_{1\xi})Y_{\eta}]-\eta (A^k+\mu_k h)M_{14}Y_{\eta}\\
    			&\quad+2Y_{\eta}Y_{\eta\eta}(\omega^k(1+M_{14}kh)-Y)(1+M_{14}kh)-(A^k+\mu_k h)r^k\\
    			&\quad-C^k\omega^k+\frac{3m-1}{2}aY(1+M_{14}kh).
    		\end{aligned}	
    	\end{equation}
    	Where
    	\begin{equation}
    		(L(Y))_{\eta}=\nu Y^2Y_{\eta\eta\eta}+2\nu YY_{\eta}Y_{\eta\eta}+(\eta^2-1)maY_{\eta\eta}+2\eta maY_{\eta}-\eta \frac{3m-1}{2}aY_{\eta}-\frac{3m-1}{2}aY=0.
    	\end{equation}
    	It follows that $|Y^2Y_{\eta\eta\eta}|\leq C_3|Y_\eta|$. And we have
    	\begin{align*}
    			&\left|(1+M_{14}kh)[((\omega^k)^2-Y^2)Y_{\eta\eta\eta}+(\eta^2-1)(ma^k_1+kha^k_{1\xi})Y_{\eta\eta}\right.\\
    			&\left. +\eta(2ma^k_1+2kha^k_{1\xi}-\frac{3m-1}{2}a^k_1-kha^k_{1\xi})Y_{\eta}]\right|\leq C_4|Y_{\eta}|kh,\\
    			&2Y_{\eta}Y_{\eta\eta}(\omega^k(1+M_{14}kh)-Y)(1+M_{14}kh)\geq 2C_5|Y_{\eta}|(M_{14}-M_{11})kh,\\
    			&-(A^k+\mu_k h)r^k-C^k\omega^k+\frac{3m-1}{2}aY(1+M_{14}kh)\geq C_6Y(M_{14}-C_7)kh.
    	\end{align*}
    	
    	Then $P_k(F_1)>0$ for $0<\eta<1$ and $kh\leq X_4$, where $X_4$ is sufficiently small, if $M_{14}$ is chosen large enough.
    	
    	It follows the first inequality of (\ref{ine}) that there exist sequences $\eta'_n\rightarrow 1, \eta''_n\rightarrow 1 (n\rightarrow\infty)$ such that 
    	\begin{equation}
    		\begin{gathered}
    			z^k|_{\eta=\eta'_n}\leq Y_{\eta}(\eta'_n)(1-M_{11}kh),\\
    			z^k|_{\eta=\eta''_n}\geq Y_{\eta}(\eta''_n)(1+M_{12}kh).
    		\end{gathered}
    	\end{equation}
     	Let $M_{14}>M_{12}$. Consider the function $y^k=z^k-F^k_1$. Then we have $y^k|_{\eta=\eta''_n}> 0$. It is easy to see that
    	\begin{equation}
    		y^k(0)=\left.\left[-Y_\eta M_{14}kh+\frac{v_1^k}{\nu}+\frac{1}{\nu}\left(\frac{ma}{Y}-\frac{B^k}{\omega^k}\right)\right]\right|_{\eta=0}\geq 0
    	\end{equation}
        for large enough $M_{14}$.
    	
    	For $0<\eta<1$, the function $y^k$ satisfies the following relation:
    	\begin{equation}
    		\begin{aligned}
    			P_k(z)-P_k(F_1)=&\nu (\omega^k)^2y^k_{\eta\eta}+(\eta^2-1)B^ky^k_{\eta}+\eta(2B^k-C^k)y^k\\
    			&-\eta (A^k+\mu_k h)\frac{y^k-y^{k-1}}{h}+2\omega^kz^ky^k_{\eta}+2\omega^kF^k_{1\eta}y^k<0.
    		\end{aligned}
    	\end{equation}
    	We see that the coefficient of $y^k$ is negative since
    	\begin{equation}
    		\eta(2B^k-C^k)-\eta\frac{(A^k+\mu_k h)}{h}+2\omega^kF^k_{1\eta}\leq \eta\frac{m+1}{2}a-\eta ka-\eta \mu_k+\eta N_4 kh\leq 0.
    	\end{equation}
        We need to choose $\mu_k$ large enough when $m\geq 1$. Therefore $y^k\geq 0$, i.e. $\omega^k_{\eta}\geq Y_{\eta}(1+M_{14}kh)$, for $0<\eta<1$, since $\eta''_n\rightarrow 1$ as $n\rightarrow \infty$. In a similar way, we show that 
    	\begin{equation}
    		\omega^k_{\eta}\leq Y_{\eta}(1-M_{15}kh)
    	\end{equation}
    	for $kh\leq X_5$ and $0<\eta<1$, where $M_{15}$ is sufficiently large and independent of $h,k$. 
    	
    	Next, let us estimate $z^k=w^k_{\eta}$ when $0<m< \frac{1}{3}$. Setting $G_1^k=-M_{16}\sigma$, we find that
    	\begin{equation}
    		\begin{aligned}
    			P_k(G_1)=&-M_{16}\nu(\omega^k)^2\left(\frac{1}{2\sigma(1-\eta)^2}-\frac{1}{4\sigma^3(1-\eta)^2}\right)-M_{16}(\eta^2-1)B^k\frac{1}{2\sigma(1-\eta)}\\
    			&\quad -M_{16}\eta (2B^k-C^k)\sigma+\nu\omega^kM_{16}^2\frac{1}{1-\eta}-(A^k+\mu_kh)r^k-C^k\omega^k\\
    			\geq& -\frac{\nu M_{16}M_6^2\sigma}{2}(1+M_{12}kh)^2+\frac{\nu M_{16}M_6^2}{4\sigma}(1+M_{12}kh)^2+M_{16}(1+\eta)\frac{B^k}{2\sigma}\\
    			&\quad-M_{16}\eta (2B^k-C^k)\sigma+\nu M_{16}^2M_5\sigma(1-M_{11}kh)-(A^k+\mu_kh)r^k-C^k\omega^k\\
    			>&0.
    		\end{aligned}
    	\end{equation}
        for $kh\leq X_6$, provided that $M_{16}$ is suitably large. Let us choose $M_{16}$ to be independent of $k,h$ and such that $z^k-G^k_1\geq 0$ for $\eta=\eta''_n$, and moreover, $z^k-G^k_1\geq 0$ for $\eta=0$.
        
        For the difference $s_1^k=z^k-G^k_1$, $0\leq \eta<1$, we obtain the following inequality
        \begin{equation}
        	\begin{aligned}
        		&\nu (\omega^k)^2s^k_{1\eta\eta}+(\eta^2-1)B^ks^k_{1\eta}+\eta(2B^k-C^k)s_1^k\\
        		&-\eta (A^k+\mu_k h)\frac{s_1^k-s_1^{k-1}}{h}+2\omega^kz^ks^k_{1\eta}+2\omega^kG^k_{1\eta}s_1^k<0.
        	\end{aligned}
        \end{equation}
        Moreover,
        \begin{equation}
        	s_1^k(\eta''_n)\geq 0,\quad s_1^k(0)\geq 0.
        \end{equation}
    
        The coefficient of $s_1^k$ is equal to
        $$
        \eta(2B^k-C^k)-\eta \frac{A^k+\mu_k h}{h}+2\omega^kG^k_{1\eta}<0.
        $$
        Therefore $s_1^k\geq 0$, i.e. $\omega^k_{\eta}\geq -M_{16}\sigma$, for $0<\eta<1$, since $\eta''_n\rightarrow 1$ as $n\rightarrow \infty$.
        
    	Now, let us estimate $z^k$ from above.  It follows from (\ref{line}) that
    	\begin{equation}
    		\Lambda_k(z)=\nu (\omega^k)^2z^k_\eta+(\eta^2-1)B^kz^k=\eta(A^k+\mu_kh)r^k+\eta C^k\omega^k\leq 0.
    	\end{equation}
        Therefore, it is easy to obtain $z^k$ is bounded from above for $0\leq \eta\leq 1-\delta, \forall\delta>0$ by the Gronwall theorem, as well as from the estimate for $z^k$ at $\eta=0$ obtained from the boundary condition (\ref{bd6}). This is possible, since
        \begin{equation}
        	\begin{aligned}
        		z^k(\eta)&=\left(z^k(0)+\int_{0}^{\eta}J(\tau)e^{\int_{0}^{\tau}I(s)ds}d\tau\right)e^{-\int_{0}^{\eta}I(s)ds}\\
        		&\leq z^k(0),
        	\end{aligned}
        \end{equation} 
        where we have
        \begin{equation}
        	\begin{gathered}
        		I(\eta)=\frac{(\eta^2-1)B^k}{\nu (\omega^k)^2}\leq 0,\\
        		J(\eta)=\frac{\eta(A^k+\mu_kh)r^k+\eta C^k\omega^k}{\nu (\omega^k)^2}\leq 0.
        	\end{gathered}
        \end{equation}
        
        Next, let us estimate $z^k$ from above in the neighborhood of $\eta=1$. Setting $G^k_2=-M_{20}\sigma$, we obtain the following inequality
    		\begin{align*}
    			P_k(G_2)&=-M_{20}\nu(\omega^k)^2\left(\frac{1}{2\sigma(1-\eta)^2}-\frac{1}{4\sigma^3(1-\eta)^2}\right)-M_{20}(\eta^2-1)B^k\frac{1}{2\sigma(1-\eta)}\\
    			&\quad -M_{20}\eta (2B^k-C^k)\sigma+\nu\omega^kM_{20}^2\frac{1}{1-\eta}-(A^k+\mu_kh)r^k-C^k\omega^k\\
    			\leq& -\frac{\nu M_{20}M_5^2\sigma}{2}(1-M_{11}kh)^2+\frac{\nu M_{20}M_5^2}{4\sigma}(1-M_{11}kh)^2+M_{20}(1+\eta)\frac{B^k}{2\sigma}\\
    			&\quad-M_{20}\eta (2B^k-C^k)\sigma+\nu M_{20}^2M_6\sigma(1+M_{12}kh)-(A^k+\mu_kh)r^k-C^k\omega^k\\
    			<&0.
    		\end{align*}
        for $1-\delta_1 \leq \eta<1$ and $kh\leq X_7$, provide that $M_{20}$ and $\delta_1$ are sufficiently small ($M_{20}$ and $\delta_1$ are independent of $k,h$). 
        
        Consider the difference $s_2^k=z^k-G^k_2$. For $1-\delta_1 \leq \eta<1$ and $kh\leq X_7$, the function $s_2^k$ satisfies the following inequalities
        \begin{equation}
        	\begin{aligned}
        		&\nu (\omega^k)^2s^k_{2\eta\eta}+(\eta^2-1)B^ks^k_{2\eta}+\eta(2B^k-C^k)s_2^k\\
        		&\quad-\eta (A^k+\mu_k h)\frac{s_2^k-s_2^{k-1}}{h}+2\omega^kz^ks^k_{2\eta}+2\omega^kG^k_{2\eta}s_2^k>0.\\
        		&s_2^k\leq 0,\quad \text{for}\quad \eta=1-\delta_1,\\
        		&s_2^k\leq 0,\quad \text{for}\quad \eta=\eta'_n.
        	\end{aligned}
        \end{equation}
    
        Note the coefficient of $s_2^k$ is negative. Then $s_2^k\leq 0$, i.e. $\omega^k_{\eta}\leq -M_{20}\sigma$, for $1-\delta \leq \eta<1$, since $\eta'_n\rightarrow 1$ as $n\rightarrow \infty$. Choosing suitable $M_{17}$, we have $\omega^k_{\eta}\leq -M_{17}\sigma$, for $0 \leq \eta<1$ and $kh\leq X_7$.
    
    	It follows from (\ref{line}) that
    	\begin{equation}
    		\nu \omega^k\omega^k_{\eta\eta}=\eta C^k+(1-\eta^2)B^k\frac{\omega^k_{\eta}}{\omega^k}+\eta (A^k+\mu_k h)\frac{r^k}{\omega^k}.
    	\end{equation}
    	And we have
    	\begin{equation}
    		\begin{aligned}
    			|\omega^k\omega^k_{\eta\eta}|&\leq M_{18},\\
    			\nu \omega^k\omega^k_{\eta\eta}&\leq \eta\frac{3m-1}{2}a+(1-\eta^2)ma\frac{\omega^k_{\eta}}{\omega^k}+N_5 kh\\
    			&\leq - M_{19},\quad\text{for}\quad 0<m<\frac{1}{3},\\
    			\nu \omega^k\omega^k_{\eta\eta}&=(1-\eta^2)ma\frac{Y_\eta}{Y}+\eta C^k+\eta (A^k+\mu_k h)\frac{r^k}{\omega^k}+(1-\eta^2)ma\left(\frac{\omega_{\eta}^k}{\omega^k}-\frac{Y_{\eta}}{Y}\right)\\
    			&\leq \nu YY_{\eta\eta}+N_6kh\\
    			&\leq -M_{19},\quad\text{for}\quad m\geq\frac{1}{3}.
    		\end{aligned}
    	\end{equation}
    	for $kh\leq X_8$ and $0\leq \eta<1$, where $X_8$ is sufficiently small. Thus, in order to complete the proof, it suffices to take $\mathop{min}\limits_{1 \leq i \leq 8} X_i$.

    \end{proof} 

    Next, on the basis of Lemma \ref{wk}, we establish the following existence theorem the problem (\ref{qy}), (\ref{bd3}).
    
    \begin{theorem}\label{nonline}
    	Assume that $U(x)$ and $v_0(x)$ satisfy the conditions (\ref{assum}). Then, in the domain $\varOmega$, with $X$ depending on $U(x)$ and $v_0(x)$, problem (\ref{qy}),(\ref{bd3}) has a unique classical solution $\omega(\xi,\eta)$ which is positive for $\eta<1$ and has the following properties: 
    	\begin{equation}
    		\begin{gathered}
    			Y(\eta)(1-M_{11}\xi)\leq \omega \leq Y(\eta)(1+M_{12}\xi),\\
    			|\omega_{\xi}(\xi,\eta)|\leq M_{13}Y(\eta),\\
    			Y_{\eta}(1+M_{14}\xi)\leq \omega_{\eta} \leq Y_{\eta}(1-M_{15}\xi),\quad \text{for} \quad m\geq \frac{1}{3},\\
    			-M_{16} \sigma\leq \omega_{\eta} \leq -M_{17} \sigma,\quad \text{for}\quad 0<m<\frac{1}{3},\\
    			-M_{18}\leq \omega \omega_{\eta\eta}\leq -M_{19},
    		\end{gathered}
    	\end{equation}
    	where $\sigma(\eta)=\sqrt{-\ln \mu(1-\eta)}$, $\mu=\text{const}$, $0<\mu<1$.
    \end{theorem}

    	\begin{proof}
    	
    	The functions $\omega^k(\eta) \equiv \omega(kh, \eta)$ defined as solutions of system (\ref{line}), (\ref{bd6}) can be extended in $\xi$ for $kh<\xi \leq(k+1) h$ as linear functions. Thus we obtain a family $\omega_h(\xi, \eta)$ such that
    	\begin{equation}
    		\begin{aligned}
    			\omega_h(kh(1-\lambda)+(k+1)h\lambda,\eta)&=\omega^k(\eta)(1-\lambda)+\omega^{k+1}(\eta)\lambda,\\
    			0\leq\lambda\leq 1,\quad k&=0,1,2,\cdots.
    		\end{aligned}
    	\end{equation}
    	
    	According to Lemma \ref{wk}, the functions $\omega_h(\xi, \eta)$ from this family satisfy the Lipschitz condition with respect to $\xi$ and, for $0 \leq \xi \leq X$, $0 \leq \eta \leq 1-\varepsilon, 0<\varepsilon<1$, have uniformly (in $h$ ) bounded first derivative in $\eta$. By the Arzel\`{a} theorem, there is a sequence $h_i \rightarrow 0$ such that $\omega_{h_i}$ uniformly converge, on the rectangle $\{0 \leq x \leq X, 0 \leq \eta \leq 1-\varepsilon\}$, to some $\omega(\xi, \eta)$. Since $Y(1-M_{11}kh) \leq \omega_h \leq Y(1+M_{12}kh)$, the sequence $\omega_{h_i}$ can be chosen to be uniformly convergent to $\omega$ in $\varOmega$.
    	
    	It follows from (\ref{ine}) that $\omega(\xi, \eta)$ has bounded weak derivatives $\omega_{\xi}, \omega_\eta, \omega_{\eta \eta}$ and that $\omega \omega_{\eta \eta}$ is bounded in $\Omega$. The sequence $\omega_{h_i}(\xi, \eta)$ may be assumed such that the derivatives $\omega_{\xi}, \omega_\eta, \omega_{\eta \eta}$ in the domain
    	$$
    	\varOmega_{\varepsilon}=\{0<x<X, 0<\eta<1-\varepsilon\}, \quad \varepsilon=\text { const }>0,
    	$$
    	coincide with weak limits in $L_2\left(\varOmega_{\varepsilon}\right)$ of the respective functions
    	$$
    	h_i^{-1}\left(\omega_{h_i}\left(\xi+h_i, \eta\right)-\omega_{h_i}(\xi, \eta)\right), \quad \omega_{h_i \eta}, \omega_{h_i \eta \eta} .
    	$$
    	
    	Let us show that equation (\ref{qy}) holds for $\omega(\xi, \eta)$ everywhere in $\varOmega$. By construction, the function $\omega_h^k(\eta)$ satisfies the equation (\ref{line}), i.e.,
    	\begin{equation}\label{wh}
    		\begin{aligned}
    			\nu (\omega^k_h)^2\omega^k_{h\eta\eta}-\eta (A^k+\mu_k h)\frac{\omega^k_h-\omega^{k-1}_h}{h}+(\eta^2-1)B^k\omega^k_{h\eta}-\eta C^k\omega^k_h=0,\\
    			k=1,2,\cdots, m(h), \quad m(h)=[X / h].
    		\end{aligned}
    	\end{equation}
    	
    	Let $\varphi(\xi, \eta)$ be an infinitely differentiable function with compact support strictly inside $\varOmega$. Multiplying (\ref{wh}) by $h \varphi^k(\eta) \equiv h \varphi(k h, \eta)$, integrating the resulting equation in $\eta$ from $0$ to $1$ , and taking the sum over $k$ from $1$ to $m(h)$, we obtain
    	\begin{equation}\label{int}
    		\begin{gathered}
    			\sum_{k=1}^{m(h)} h \int_0^1 \left[\nu (\omega^k_h)^2\omega^k_{h\eta\eta}\varphi^k-\eta (A^k+\mu_k h)\left(\frac{\Delta\omega_h}{h}\right)^k\varphi^k+(\eta^2-1)B^k\omega^k_{h\eta}\varphi^k-\eta C^k\omega^k_h\varphi^k\right]d\eta=0,\\
    			\quad \left(\frac{\Delta\omega_h}{h}\right)^k=\frac{\omega^k_h-\omega^{k-1}_h}{h}.
    		\end{gathered}
    	\end{equation}
    	
    	Denote by $\bar{f}(\xi, \eta)$ a function defined in $\varOmega$ and equal to $f^k(\eta)=f(k h, \eta)$ for $(k-1) h<\xi \leq k h, k=1,2, \cdots, m(h)$. Using this notation, we can rewrite (\ref{int}) in the form
    	\begin{equation}\label{bar}
    		\int_{\varOmega}\left[\nu (\overline{\omega}_h)^2\overline{\omega}_{h\eta\eta}\overline{\varphi}-\eta \overline{A\varphi}\overline{\left(\frac{\Delta\omega_h}{h}\right)}
    		+(\eta^2-1)\overline{B\varphi}\overline{\omega}_{h\eta}-\eta \overline{C\varphi}\overline{\omega}_h\right]d\xi d\eta=0.
    	\end{equation}

    	The functions $\overline{\omega}, \overline{A\varphi}, \overline{B\varphi}, \overline{C\varphi}$ are uniformly convergent to $\omega, A \varphi, B \varphi, C\varphi$, respectively, as $h_i \rightarrow 0$.
    	
    	For any $\varepsilon\in(0,1)$, the functions $\overline{\left(\frac{\Delta\omega_h}{h}\right)}, \overline{\omega}_{h\eta}, \overline{\omega}_{h\eta\eta}$ are weakly convergent in $L_2(\varOmega_{\varepsilon})$ to $\omega_{\xi}, \omega_{\eta}, \omega_{\eta \eta}$, respectively, as $h_i \rightarrow 0$. Passing to the limit as $h_i \rightarrow 0$ in (\ref{bar}), we find that
    	\begin{equation}
    		\int_{\varOmega}\left(\nu\omega^2\omega_{\eta \eta}-\eta A\omega_{\xi}+(\eta^2-1)B\omega_{\eta}-\eta C\omega\right)\varphi d\xi d\eta =0,
    	\end{equation}
    	which implies, owing to our choice of $\varphi$, that equation (\ref{qy}) holds almost everywhere for $\omega(\xi,\eta)$ in $\varOmega$.
    	
    	It follows from (\ref{qy}) that the function $\omega(\xi,\eta)$ and its derivatives in (\ref{qy}) satisfy the H\"{o}lder condition in any domain strictly interior to $\varOmega$, since the Lipschitz condition holds for $\omega$ and $\omega>0$.
    	
    	Let us show that the boundary conditions (\ref{bd3}) are also satisfied. The validity of the first condition follows from the uniform convergence of $\omega_h$ and (\ref{bd6}) for $\eta=1$. Set 
    	\begin{equation}
    		\beta_h(\xi,\eta)=\nu \overline{\omega}_{h\eta}-\bar{v_1}+\frac{\bar{B}}{\bar{\omega}_h}.
    	\end{equation}
    	Since $\bar{\omega}_{h\eta\eta}$, for $0\leq \eta<1-\delta, \delta=\text{const}>0$, is bounded uniformly in $h$ and the second boundary condition in (\ref{bd6}) is satisfied, we have 
    	\begin{equation}
    		\beta_h(\xi,0)=0,\quad |\beta_h(\xi,\eta)|\leq C_8\eta,\quad 0\leq\eta<1-\delta.
    	\end{equation}
    	
    	The function $\beta=\nu \omega_{\eta}-v_1+\frac{B}{\omega}$ is a weak limit in $L_2(\varOmega)$ of $\beta_h$ as $h_i\rightarrow 0$. Therefore, $|\beta(\xi,\eta)|\leq C_8\eta$, which means that the boundary condition (\ref{bd3}) is satisfied.
    	
    	Next, let us prove the uniqueness of the solution. Assume that problem (\ref{qy}), (\ref{bd3}) in $\varOmega$ admits two solutions, say $\omega_1,\omega_2$, with the above properties. Set $\bar{\omega}=\omega_1-\omega_2$. Then
    	\begin{equation}
    		\begin{gathered}
    			\nu \omega^2_1\bar{\omega}_{\eta\eta}-\eta A\bar{\omega}_{\xi}+(\eta^2-1)B\bar{\omega}_{\eta}-\eta C\bar{\omega}+\nu (\omega_1+\omega_2)\omega_{2\eta\eta}\bar{\omega}=0 \quad \text{in}\quad \varOmega,\\
    			\bar{\omega}(1)=0,\quad \left.\left(\nu \bar{\omega}_{\eta}-\frac{B}{\omega_1\omega_2}\bar{\omega}\right)\right|_{\eta=0}=0.
    		\end{gathered}	
    	\end{equation}
    	We have 
    	\begin{equation}
    		\begin{aligned}
    			&-\eta C+\nu (\omega_1+\omega_2)\omega_{2\eta\eta}\\
    			\leq &-\eta \frac{3m-1}{2}a+\eta\frac{3m-1}{2}a+(1-\eta^2)ma\frac{\omega^k_{\eta}}{\omega^k}+N_7kh<0
    		\end{aligned}
    	\end{equation}
    	for $0< \eta< 1$ when $0<m<\frac{1}{3}$, since
    	\begin{equation}
    		\nu \omega^k_2\omega^k_{2\eta\eta}\leq \eta\frac{3m-1}{2}a+(1-\eta^2)ma\frac{\omega^k_{\eta}}{\omega^k}+N_5 kh.
    	\end{equation}
        And we have
        \begin{equation}
        	-\eta C+\nu (\omega_1+\omega_2)\omega_{2\eta\eta}<0
        \end{equation}
        for $0< \eta< 1$ as $m\geq\frac{1}{3}$.
    	
    	Therefore, $\bar{\omega}$ can have neither a positive maximum nor a negative minimum. Consequently, $\bar{\omega}\equiv 0$, $\omega_1\equiv \omega_2$.

    \end{proof}
    
    As an immediate corollary, we obtain Theorem \ref{prandtl} by making the inverse change of variables (\ref{crocco}).
    
    \begin{proof}{Proof of Theorem \ref{prandtl}}
    	
    	According to (\ref{crocco}), we have
    	\begin{equation}
    		x=\xi,\quad y=\int_{0}^{u/U}\frac{1}{x^{\frac{m-1}{2}}\omega(x,s)}ds.
    	\end{equation}
    	Hence, by virtue of the continuity of $\omega(\xi, \eta)$ in $\bar{\varOmega}$ and the inequality $\omega>0$ for $0 \leq \eta<1$, we see that $u(x, y) / U(x)$ is continuous in $\bar{D}$,
    	\begin{equation}
    		u(x, 0)=0, \quad u(0, y)=0, \quad u(x, y) \rightarrow U(x) \text { as } y \rightarrow \infty,
    	\end{equation}
    	$u_y$ is bounded and continuous in $\bar{D}$, $u_y>0$ for $y \geq 0, x \geq 0$. And we find that
    	\begin{equation}
    		\begin{aligned}\label{rel}
    			&u_{yy}=x^{\frac{m-1}{2}}\omega_{\eta}u_y,\quad u_{yyy}=x^{\frac{m-1}{2}}(\omega_{\eta\eta}\frac{u^2_y}{U}+\omega_{\eta}u_{yy}),\\
    			&u_{yx}=(x^{\frac{m-1}{2}}U)_x\omega+x^{\frac{m-1}{2}}U(\omega_{\xi}+\omega_{\eta}(\frac{u_x}{U}-\frac{uU_x}{U^2})),\\
    			&u_x=u\frac{U_x}{U}+\omega U\int_{0}^{u/U}\left(\frac{\omega_{\xi}(x,s)}{\omega^2(x,s)}+\frac{m-1}{2}\frac{1}{x\omega(x,s)}\right)ds.
    		\end{aligned}
    	\end{equation}
    	
    	From the properties of $\omega$ and its derivatives, in combination with (\ref{rel}), it follows that the generalized derivatives $u_x, u_{y y}, u_{y y y}$ are bounded in $D$, and $u_{x y}$ is bounded for finite $y$. The first inequality for $u$ follows from the estimates for $\omega$. The continuity of $u_x$ and $u_{y y}$ with respect to $y$ follows from (\ref{rel}). Let us define $v(x, y)$ by
    	\begin{equation}
    		v=\frac{-u u_x+\nu u_{y y}+U U_x}{u_y}.
    	\end{equation}
    	
    	The function $v$ has the first derivative with respect to $y$ in $D$, we obtain the equations
    	\begin{equation}\label{u1}
    		v_y u_y+u_y u_x+\frac{u_{y y}}{u_y}(-u u_x+U U_x+\nu u_{y y})+u u_{x y}-\nu u_{y y y}=0.
    	\end{equation}
    	
    	The function $w(\xi, \eta)=\frac{u_y}{x^{\frac{m-1}{2}}U}$ satisfies equation (\ref{qy}). Replacing the derivatives of $\omega$ by those of $u$ in (\ref{qy}), we find that
    	\begin{equation}\label{u2}
    		\frac{1}{x^{\frac{3(m-1)}{2}}U}\left\{-uu_{xy}+uu_x\frac{u_{yy}}{u_y}-UU_x\frac{u_{yy}}{u_y}+\nu\frac{u_yu_{yyy}-u^2_{yy}}{u_y}\right\}=0.
    	\end{equation}
    	
    	Multiplying (\ref{u2}) by $x^{\frac{3(m-1)}{2}}U$ and adding the result to (\ref{u1}), we obtain
    	\begin{equation}
    		v_y u_y+u_y u_x=0,
    	\end{equation}
    	or equivalently,
    	\begin{equation}
    		u_x+v_y=0.
    	\end{equation}
    	
    	Let us show that $v(x, y)$ satisfies the condition
    	\begin{equation}
    		v(x, 0)=v_0(x)=x^{\frac{m-1}{2}}v_1(x).
    	\end{equation}
    	It follows from (\ref{bd3}) that
    	\begin{equation}
    		v_1=\left.\left(\frac{\nu \omega \omega_{\eta}+B}{\omega}\right)\right|_{\eta=0}.
    	\end{equation}
    	And we find that
    	\begin{equation}
    		v(x,0)=\left.\left[\frac{\nu u_{y y}+U U_x}{u_y}\right]\right|_{y=0}=x^{\frac{m-1}{2}}\left.\left[\frac{\nu \omega\omega_{\eta}+ B}{\omega}\right]\right|_{\eta=0}=v_0(x).
    	\end{equation}
    	It follows that the function $v$ is continuous in $\bar{D}$ with respect to $y$, and is bounded for bounded $y$, $v_y$ is bounded in $D$.
    	
    	Using the estimates for $\omega(\xi,\eta)$ established in Theorem \ref{nonline} and those for $Y(\eta)$ from Theorem \ref{Y}, we find that the following inequality is valid:
    	\begin{equation}
    		M_5(1-\eta)\sigma(1-M_{11}\xi) \leq \omega(\xi,\eta) \leq M_6(1-\eta)\sigma(1+M_{12}\xi).
    	\end{equation}
    	This brings us to the inequalities as $y\rightarrow \infty$
    	\begin{equation}
    		\frac{\sigma}{M_6x^{\frac{m-1}{2}}(1+M_{12}x)}\leq\frac{2(\sigma-\sigma(0))}{M_6x^{\frac{m-1}{2}}(1+M_{12}x)}\leq y\leq\frac{2(\sigma-\sigma(0))}{M_5x^{\frac{m-1}{2}}(1-M_{11}x)}\leq\frac{2\sigma}{M_5x^{\frac{m-1}{2}}(1-M_{11}x)}
    	\end{equation}
    	since
    	\begin{equation}
    		y=\int_{0}^{u/U}\frac{1}{x^{\frac{m-1}{2}}\omega(x,s)}ds.
    	\end{equation}
    	And we find that
    	\begin{equation}
    		\frac{1}{2}M_5x^{\frac{m-1}{2}}(1-M_{11}x)y \leq \sigma \leq M_6x^{\frac{m-1}{2}}(1+M_{12}x)y,
    	\end{equation}
    	and therefore
    	\begin{equation}
    		-M^2_6x^{m-1}y^2(1+M_{12}x)^2 \leq \ln \mu(1-\frac{u}{U}) \leq -\frac{1}{4}M^2_5x^{m-1}y^2(1-M_{11}x)^2.
    	\end{equation}
    	Hence we see that
    	\begin{equation}
    		M_1 \exp (-M_2 x^{m-1} y^2) \leq 1-\frac{u}{U} \leq M_3 \exp (-M_4 x^{m-1} y^2).
    	\end{equation}
    	
    	If $u, v$ is a solution of problem (\ref{ps}),(\ref{bd}) with the properties listed in Theorem \ref{prandtl}, then, changing the variables by (\ref{crocco}) and introducing the function $\omega=\frac{u_y}{x^{\frac{m-1}{2}}U}$, we arrive at a solution $\omega$ of problem (\ref{qy}), (\ref{bd3}) with the properties specified in Theorem \ref{nonline}. As shown above, the latter solution is unique.
    	
    \end{proof}

    \section{Axially symmetric stationary boundary layer}
    
    In this section, we consider the following Prandtl system for the axially symmetric three-dimensional incompressible stationary flow:
    \begin{equation}\label{ps2}
    	\begin{gathered}
    		u\frac{\partial u}{\partial x}+v\frac{\partial u}{\partial y}=U\frac{dU}{dx}+\nu\frac{\partial^2 u}{\partial y^2};\\
    		\frac{\partial (ru)}{\partial x}+\frac{\partial (rv)}{\partial y}=0
    	\end{gathered}
    \end{equation}
    in the domain $D=\{0<x<X,0<y<\infty\}$ with the boundary conditions 
    \begin{equation}\label{boundary}
    	\begin{gathered}
    		u(0,y)=0,u(x,0)=0,v(x,0)=v_0(x),\\
    		u\rightarrow U(x) \quad \text{as} \quad y\rightarrow \infty.
    	\end{gathered}
    \end{equation}
    
    The function $r(x)$ determines the surface of the body past which the fluid flows: $r(0)=0$, $r(x)>0$ for $x>0$, $r_x(0)\neq 0$. $U(x)$ is a given longitudinal velocity component of the outer flow: $U(0)=0$, $U(x)>0$ for $x>0$, $\nu$ is the viscosity coefficient.
    
    For instance, when three-dimensional axially symmetric flow flows to a cone with the cone angle $0<\phi<2\pi$, the system mentioned above appears near the surface of the cone (cf.Figure \ref{fig:axially}). It is known that the Euler flow near the surface of the cone possesses the following property. The flow velocity varies according to the law $U(x)\sim Cx^m$, where $C,m=\text{const}>0$. When$\phi=\pi$, $m=1$ (for example, the three-dimensional flow against a plane wall). When $0<\phi<\pi$, $0<m<1$ and when $\pi<\phi<2\pi$, $m>1$. Perhaps, due to the curvature of the surface, lower-order terms appear in this asymptotics.
    
    \begin{figure}[h]
    	\centering
    	\includegraphics[scale=0.6]{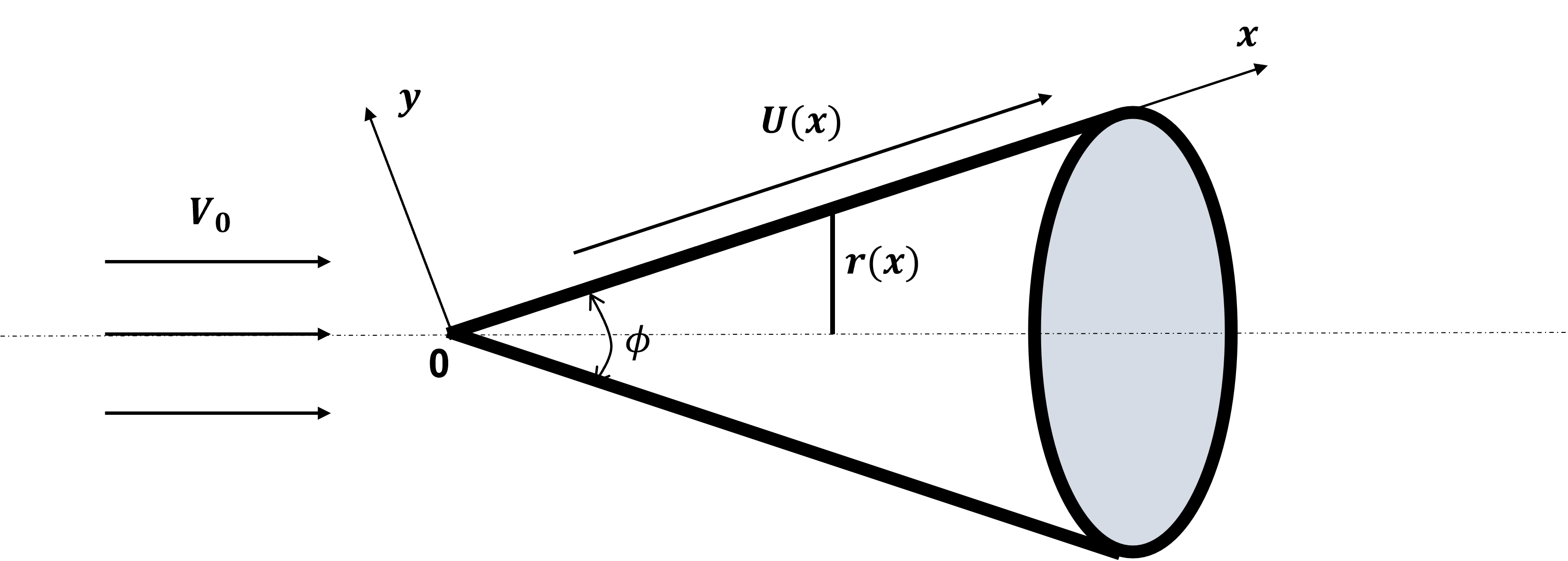}
    	\caption{Flow parallel to the symmetry axis of a cone}
    	\label{fig:axially}
    \end{figure}
    
    Oleinik and Samokhin considered the boundary layer problem near a critical point on the wall in the case of an axially symmetric flow in \cite{OS}. The main results of them can be summarized as that, when $\pi\leq\phi<2\pi$, there exists a unique classical solution to the problem (\ref{ps2}), (\ref{boundary}) for $x<X$, where $X$ is small, in the class of functions having a certain behavior at infinity with respect to $y$, namely, $U(x)-u(x,y)\sim e^{-\alpha y^2}$ as $y\rightarrow \infty$, where $\alpha=\text{const}>0$. The main purpose of this section is to establish the existence, uniqueness, and asymptotic behavior of solutions to the system (\ref{ps}), (\ref{boundary}) for arbitrary $0<\phi<2\pi$. This generalizes the local well-posedness results due to Oleinik mentioned above.
    
    The axially symmetric stationary boundary layer problem can also be solved by the line method on the basis of the Crocco transformation. The approach of the preceding sections can be applied with no principal modifications. For this reason, we restrict ourselves to the statement of the problem and the main result.
   
    In what follow, $M_i,N_i$ stand for positive constants.
    
    We consider problem (\ref{ps2}),(\ref{boundary}) under the assumption that
    \begin{equation}\label{assum2}
    	\begin{gathered}
    		U(x)=x^mV(x),\quad v_0(x)=x^{\frac{m-1}{2}}v_1(x),\quad r(x)=xr_1(x),\\
    		V(x)=a+a_1(x), a>0,\quad v_1(0)=0,\quad r_1(x)=c+c_1(x), 0<c\leq 1,\\
    		|a_1(x)|\leq N_1x,\quad |v_1(x)|\leq N_2x,\quad c_1(x)\leq N_3x.
    	\end{gathered} 
    \end{equation}
    
    Introducing a generalization of the Crocco variables
    \begin{equation}
    	\xi=x,\quad \eta=\frac{u(x,y)}{U(x)},
    \end{equation}
    we obtain for $\omega(\xi ,\eta)=\frac{u_y(x,y)}{x^{\frac{m-1}{2}}U(x)}$ the following equation:
    \begin{equation}\label{qy2}
    	\nu \omega^2\omega_{\eta\eta}-\eta A\omega_{\xi}+(\eta^2-1)B\omega_{\eta}-\eta C\omega=0,
    \end{equation}
    in the domain $\varOmega=\{0<\xi<X,0<\eta<1\}$, with the boundary conditions
    \begin{equation}\label{boundary2}
    	\left.\omega\right|_{\eta=1}=0,\quad \left.\left(\nu \omega\omega_{\eta}-v_1\omega +B\right)\right|_{\eta=0}=0,
    \end{equation}
    where 
    \begin{equation}
    	A=\xi V(\xi),\quad B=mV(\xi)+\xi V_{\xi}(\xi),\quad C=\frac{3(m-1)}{2}V(\xi)+\xi V_{\xi}(\xi)-\frac{\xi r_{1\xi}}{r_1}V(\xi).
    \end{equation}
    
    When $\xi\rightarrow$, let $Y(\eta)=\omega(0,\eta)$, the equation (\ref{qy2}) degenerates to the following elliptic equation:
    \begin{equation}\label{begin2}
    	\nu Y^2Y_{\eta\eta}+(\eta^2-1)maY_{\eta}-\eta \frac{3(m-1)}{2}aY=0,\quad 0<\eta <1,
    \end{equation} 
    and the boundary conditions become
    \begin{equation}\label{boundary3}
    	Y(1)=0,\quad (\nu YY_{\eta}+ma)|_{\eta=0}=0.
    \end{equation}
    
    As mentioned in Section II, what we care about is the following self-similarity results.
    
    Consider $U(x)=ax^m,a>0$, $v_0(x)=0$, $r(x)=cx$, $u(x,y)$ has the following self-similarity form:
    \begin{equation}
    	u(x,y)=x^mf^{'}(yx^\frac{m-1}{2}),
    \end{equation}
    and $f(z)$ satisfies the following equation
    \begin{equation}\label{si2}
    	f^{'''}+\frac{m+3}{2}aff^{''}+ma(1-f^{'2})=0,
    \end{equation}
    with the boundary conditions
    \begin{equation}
    	f(0)=0, \quad f^{'}(0)=0,\quad f^{'}(z)\rightarrow 1 \quad \text{as} \quad z\rightarrow\infty.
    \end{equation}
    Furthermore, introduce $\eta=f^{'}(z)$, $Y(\eta)=f^{''}(z)$, we can obtain the equation (\ref{begin2}), (\ref{boundary3}).
    
    Obviously, we can obtain the existence, uniqueness and asymptotic behavior of solutions for equation (\ref{si2}) similar to Theorem \ref{exi} and Theorem \ref{asy}. Furthermore, we can establish the following result.
    
    \begin{theorem}\label{Y2}
    	Problem (\ref{begin2}), (\ref{boundary3}) has one and only one solution with the following properties:
    	\begin{equation}
    		\begin{gathered}
    			M_1(1-\eta)\sigma \leq Y(\eta) \leq M_2(1-\eta)\sigma,\\
    			M_2(1-\eta)(\sigma-K)\leq Y,\quad\text{for} \quad 0<\eta_0\leq\eta<1, \\
    			-M_3\sigma \leq Y_{\eta} \leq -M_4\sigma,\\
    			-M_5\leq  YY_{\eta\eta} \leq -M_6,
    		\end{gathered}
    	\end{equation}
    	where $\sigma(\eta)=\sqrt{-\ln \mu (1-\eta)}, \mu=\text{const}, 0<\mu<1$, $\sigma>K$ for $\eta>\eta_0>0$.
    \end{theorem} 	
    
    Similar to Section III, we obtain the following existence theorem for the problem (\ref{qy2}), (\ref{boundary2}) by the line method.
    
    \begin{theorem}\label{nonline2}
    	Assume that $U(x)$, $v_0(x)$, and $r(x)$ satisfy the conditions (\ref{assum2}). Then, in the domain $\varOmega$, with $X$ depending on $V(x)$, $v_1(x)$ and $r_1(x)$, problem (\ref{qy2}),(\ref{boundary2}) has a unique classical solutions $\omega(\xi,\eta)$ which is positive for $\eta<1$ and has the following properties:
    	\begin{equation}
    		\begin{gathered}
    			Y(\eta)(1-M_7\xi)\leq \omega \leq Y(\eta)(1+M_8\xi),\\
    			|\omega_{\xi}(\xi,\eta)|\leq M_9Y(\eta),\\
    			Y_{\eta}(1+M_{10}\xi)\leq \omega_{\eta} \leq Y_{\eta}(1-M_{11}\xi),\quad \text{for}\quad m\geq 1,\\
    			-M_{12} \sigma\leq \omega_{\eta} \leq -M_{13} \sigma,\quad \text{for}\quad 0<m<1,\\
    			-M_{14}\leq \omega \omega_{\eta\eta}\leq -M_{15},
    		\end{gathered}
    	\end{equation}
    	where $\sigma(\eta)=\sqrt{-\ln \mu(1-\eta)}$, $\mu=\text{const}$, $0<\mu<1$.
    \end{theorem}
     
     As an immediate corollary, we obtain the following existence theorem for the problem (\ref{ps2}), (\ref{boundary}).
    
    \begin{theorem}\label{prandtl2}
    	Let $U(x)$, $v_0(x)$, and $r(x)$ satisfy the conditions (\ref{assum2}). Then the initial boundary value problem (\ref{ps2}),(\ref{boundary}) in the domain $D=\{0<x<X,0<y<\infty\}$, for some $X$ depending on $U$, $v_0$ and $r(x)$, has a unique classical solution $u(x,y),v(x,y)$ with the following properties: $u_y>0$ for $y \geq 0, x>0$ ; $u/U, u_y /(x^{(m-1) / 2} U(x))$ are bounded and continuous in $\bar{D}$ ; $u>0$ for $y>0$ and $x>0 $; $ u(x, y) \rightarrow U(x)$, $u_y \rightarrow 0$ as  $y \rightarrow \infty$; moreover, the solution $u(x,y)$ has the following property as $x\rightarrow 0$:
    	\begin{equation}
    		u(x,y)\sim x^m f^{'}\left(yx^{\frac{m-1}{2}}\right),
    	\end{equation}
    	where $f$ satisfies equation (\ref{si2}); and the solution $u(x,y)$ has the following inequalities as $y\rightarrow \infty$:
    	\begin{equation}
    		M_{16} \exp (-M_{17} x^{m-1} y^2) \leq 1-\frac{u}{U} \leq M_{18} \exp (-M_{19} x^{m-1} y^2).
    	\end{equation}
    \end{theorem}

    \begin{remark}
    	From the theorem \ref{prandtl2}, we note that a point $x_0$, $0<x_0<X$ can be found as the initial location of the continuation problem for axially symmetric boundary layer, where $u(x_0,y)>0$ and $U(x_0)>0$. Consider the system:
    	\begin{equation}
    		\begin{gathered}
    			u\frac{\partial u}{\partial x}+v\frac{\partial u}{\partial y}=U\frac{dU}{dx}+\nu\frac{\partial^2 u}{\partial y^2};\\
    			\frac{\partial (ru)}{\partial x}+\frac{\partial (rv)}{\partial y}=0
    		\end{gathered}
    	\end{equation}
    	in the domain $D=\{0<x<X,0<y<\infty\}$ with the boundary conditions 
    	\begin{equation}
    		\begin{gathered}
    			u(0,y)=u_0(y),u(x,0)=0,v(x,0)=v_0(x),\\
    			u\rightarrow U(x) \quad \text{as} \quad y\rightarrow \infty.
    		\end{gathered}
    	\end{equation}
        
        In order to solve this continuation problem by the line method on the basis of the Crocco transformation, the approach of continuation of the boundary layer (see Oleinik and Samokhin \cite{OS}) can be applied with no principal modifications. Thus we can naturally obtain the existence of the Prandtl system (\ref{ps2}), (\ref{boundary}) for any $X>0$ in the case of favorable condition either $U_x\geq 0$ and $v_0(x)\leq 0$ or $U_x>0$, and $\frac{r_x}{r}\geq \frac{U_x}{U}$. 
    \end{remark}

    \begin{acknowledgement}
    	L. Zhang is supported by NSFC under grant No.12031012 and National Key R$\&$D Program of China, Grant No.2021YFA1000800. C. Zhao is supported by National Key R$\&$D Program of China, Grant No.2021YFA1000800.
    \end{acknowledgement}

\end{document}